\documentclass{amsart}

 \usepackage{amsmath,amsfonts,amsthm,amssymb,graphics,graphicx, verbatim}

 \newtheorem{thm}{Theorem}[section]
 \newtheorem{defin}[thm]{Definition}
 \newtheorem{cor}[thm]{Corollary}
 \newtheorem{lem}[thm]{Lemma}
 
 \newtheorem{prop}[thm]{Proposition}
 \newtheorem{example}[thm]{Example}
\newtheorem{remark}[thm]{Remark}

\newcommand{\norm}[1]{\left|\!\left|{#1}\right|\!\right|}
 \newcommand{\R}{\ensuremath{\mathbb{R}}}

 \newcommand\tlide[1]{\tilde{#1}}

\newcommand\Id{\operatorname{Id}}

\title[The quantisation of normal velocity does not concentrate on hypersurfaces]{The quantisation of normal velocity does not concentrate on hypersurfaces}

\author{Melissa Tacy}

\email{mtacy@maths.otago.ac.nz}

\address{Department of Mathematics and Statistics, University of Otago, New Zeland}

\thanks{This research was partially completed while the author was a Research Fellow at the Australian National University supported by ARC grant DP150102419}

\begin{document}

  \begin{abstract}
 We seek to extend  work by Christianson-Hassell-Toth \cite{CHT} on restrictions of Neumann data of Laplacian eigenfunctions to interior hypersurfaces to a general semiclassical setting. In the semiclassical regime the appropriate generalisation is to study the restrictions of the function $v=\nu(x,hD)u$ where $\nu(x,hD)$ is the operator defined by quantising the normal velocity observable. For the Laplacian $\nu(x,hD)=\frac{1}{2}hD_{\nu}$ where $\nu$ is the normal to the hypersurface. We find that $\norm{\nu(x,hD)u}_{L^{2}(H)}\lesssim\norm{u}_{L^{2}(M)}$ provided $u$ is an $O_{L^{2}}(h)$ quasimode of the semiclassical pseudodifferential operator $p(x,hD)$. This statement should be interpreted as a statement of non-concentration for the quantisation of normal velocity.
 \end{abstract}
\maketitle

Consider a Dirichlet eigenfunction $u$ of the Laplace-Beltrami operator on a smooth Riemannian manifold $(M,g)$, that is
$$\begin{cases}
-\Delta_{g}u=\lambda^{2}u&\mbox{in}\;M\\
u|_{\partial{}M}=0.&\end{cases}$$
Rellich \cite{R40}, Bardos-Lebeau-Rauch \cite{BLR}, G{\'e}rard-Leichtnam \cite{GL93}, and Hassell-Tao \cite{HTao02, HTao10} showed that the Neumann boundary data is bounded. That is
\begin{equation}\norm{\lambda^{-1}\partial_{\nu}u}_{L^{2}(\partial{}M)}\lesssim\norm{u}_{L^{2}(M)}\label{boundaryest}\end{equation}
where $\nu$ is the normal to the the boundary $\partial{}M$. One may then naturally ask whether \eqref{boundaryest} continues to hold for interior hypersurfaces. By considering the function $v(t,x)=e^{i\lambda{}t}u(x)$ as  solution to the wave equation we can see from Tataru \cite{tataru98} that this is indeed the case. More recently Christianson, Hassell and Toth \cite{CHT} obtained the equivalent  estimate for eigenfunctions of semiclassical operators of the form $(h^{2}\Delta+V(x))$ restricted to interior hypersurfaces $H\subset{}M$, that is 
$$\norm{h\partial_{\nu}u}_{L^{2}(H)}\lesssim\norm{u}_{L^{2}(M)}.$$
 This estimate should be seen as a statement of non-concentration. Note that by Burq-G\'{e}rard-Tvetkov \cite{BGT} we know that there are eigenfunctions $u$ (highest weight spherical harmonics) with very high $L^{2}$ mass on $H$. In particular there exist fixed constants $c_{1}$ and $c_{2}$ and a sequence of eigenfunction  $u_{\lambda}$ such that
$$c_{1}\lambda^{1/4}\norm{u_{\lambda}}_{L^{2}(M)}\leq{}\norm{u_{\lambda}}_{L^{2}(H)}\leq{}c_{2}{}\lambda^{1/4}\norm{u_{\lambda}}_{L^{2}(M)}\quad{}\text{as }\lambda\to\infty.$$
However these eigenfunctions have comparatively small, $O(\lambda^{1/2})$, normal derivative so for this class of examples
$$c_{1}\lambda^{-1/4}\leq\norm{\lambda^{-1}\partial_{\nu}u}_{L^{2}(H)}\leq{}c_{2}\lambda^{-1/4}\norm{u}_{L^{2}(M)}.$$
In this paper we move the problem into a semiclassical setting to gain some intuition from quantum-classical correspondence principles. We will state a general semiclassical result that holds for quasimodes of any semiclassical pseudodifferential operator with smooth symbol.

For a smooth symbol $p(x,\xi)$ understood to represent the total (conserved) energy of a system we define the classical flow on phase space by
\begin{equation}
\begin{cases}
\dot{x}(t)=\nabla_{\xi}p(x,\xi)\\
\dot{\xi}(t)=-\nabla_{x}p(x,\xi).\end{cases}\label{classicalflow}\end{equation}
The simplest example of such a system is that of free particle motion  given by the symbol $p(x,\xi)=|\xi|_{g}^{2}$. In the classical setting observables are given by symbols $q(x,\xi)$ defined on phase space.  We can then move to the semiclassical setting by quantising these symbols to obtain semiclassical pseudodifferntial operators
$$q(x,hD)u=Op(q(x,\xi))u=\frac{1}{(2\pi{}h)^{n}}\iint{}e^{\frac{i}{h}\langle{}x-y,\xi\rangle}q(x,\xi)u(y)d\xi{}dy.$$
The Laplace operator is obtained by quantising the symbol $p(x,\xi)=|\xi|_{g}^{2}$ and therefore is the quantisation of the energy observable of free particle motion. For a hypersurface $H=\{x\mid{}x_{1}=0\}$ with $\lambda^{-1}=h$ we may write $p(x,\xi)$ in Fermi coordinates so that
$$p(x,\xi)=\xi_{1}^{2}+q(x,\xi').$$
Therefore the operator $\lambda^{-1}\partial_{x_{1}}$ is (up to constants) the quantisation of the symbol $\partial_{\xi_{1}}p(x,\xi)$ or the quantisation of the normal velocity observable. 

A productive intuition  is to consider $u$ as being comprised of small wave packets, localised in phase space, that propagate according to the classical flow. Therefore we expect to see concentration only when the packets spend a long time trapped near the hypersurface. For free particle motion such trajectories must have small normal velocity and so a packet tracking along such a trajectory is not expected to make a large contribution to $hD_{x_{1}}u$. The large contributions come from packets moving along trajectories with normal velocity bounded below. However such packets spend little time near the hypersurface and are known not to concentrate \cite{tacy09}.  

We can of course define a classical flow given by \eqref{classicalflow} for any symbol $p(x,\xi)$ so in the semiclassical setting the analogous question is: does the quantisation of normal velocity concentrate? That is if $H$ is a smooth embedded hypersurface with normal vector $\nu(x)$ and $\nu(x,\xi)$ is given by
 $$\nu(x,\xi)=\nu(x)\cdot\nabla_{\xi}p(x,\xi)$$
 can we say that
 $$\norm{\nu(x,hD)u}_{L^{2}(H)}\lesssim{}\norm{u}_{L^{2}(M)}?$$
 In this paper we answer this question in the affirmative under the assumptions that $u$ is semiclassically localised (Definition \ref{localised}) and an $O_{L^{2}}(h)$ quasimode of $p(x,hD)$ (Definition \ref{quasimode}).

 \begin{defin}\label{localised}
 We say $u$ is semiclassically localised if there exists $\chi\in{}C_{c}(T^{\star}M)$ such that
$$u=\chi(x,hD)u+O_{\mathcal{S}}(h^{\infty}).$$
where  ${\mathcal S}$ is the space of Schwartz functions.
 \end{defin}

 \begin{defin}\label{quasimode}
 Let $u\in{}L^{2}$ we denote the quasimode error of $u$ with respect to an operator $p(x,hD)$ as
 $$E_{p}[u]=p(x,hD)u$$
 We say that $u$ is an $O_{L^{2}}(h^{\beta})$ quasimode of a semiclassical pseudodifferential operator $p(x,hD)$ if
 $$\norm{E_{p}[u]}_{L^{2}(M)}\lesssim{}h^{\beta}\norm{u}_{L^{2}(M)}.$$
  Where there is no ambiguity in $p(x,hD)$ we drop the subscript and simply write $E[u]$.
 \end{defin}
 
 The main theorem of this paper is therefore  Theorem \ref{semiclassicalthm}.

 \begin{thm}\label{semiclassicalthm}
 Let $(M,g)$ be a smooth, compact  Riemannian manifold of dimension $n$ and let $H$ be a smooth embedded interior hypersurface. Suppose $u(h)$ is a family of semiclassically localised, $O_{L^{2}}(h)$ quasimodes of  a semiclassical pseudodifferential operator $p(x,hD)$ with smooth, real symbol $p(x,\xi)$.  Then
 $$\norm{\nu(x,hD)u}_{L^{2}(H)}\lesssim{}\norm{u}_{L^{2}(M)}$$
 for $\nu(x,hD)$ the semiclassical pseudodifferential operator with symbol
 $$\nu(x,\xi)=\nu(x)\cdot{}\nabla_{\xi}p(x,\xi).$$
 \end{thm}
 
 \begin{remark}
 If $u$ is an $O_{L^{2}}(h)$ quasimode of the standard quantisation $p(x,hD)$ it is also an $O_{L^{2}}(h)$ quasimode of any other quantisation (such as the Weyl quantisation) so Theorem \ref{semiclassicalthm} holds for these quantisations too.
 \end{remark}
 
 \begin{remark}
 We will choose to work in local coordinates where $H=\{x\mid{}x_{1}=0\}$. In these coordinates
 $$\nu(x,\xi)=\partial_{\xi_{1}}p(x,\xi).$$
 \end{remark}
 
Eigenfunctions of the Laplacian can be written as solutions to the semiclassical equation $p(x,hD)u=0$ where $p(x,hD)$ is the  semiclassical pseudodifferential operator with symbol $p(x,\xi)=|\xi|_{g}-1$ and therefore fall under the scope of Theorem \ref{semiclassicalthm}. This allows us to reproduce bounds on the Neumann data for interior hypersurfaces.
 
 \begin{cor}\label{eigenfunctionthm}
Let $(M,g)$ be a smooth Riemannian manifold and $H$ a smooth embedded interior hypersurface with normal $\nu(x)$. If $u$ is an $L^{2}$ normalised  approximate Laplacian eigenfunction, that is
 $$\norm{u}_{L^{2}(M)}=1\quad\mbox{and}\quad\norm{-(\Delta_{g}-\lambda^{2})u}_{L^{2}(M)}\lesssim{}\lambda$$
 then,
 $$\norm{\lambda^{-1}\partial_{\nu}u}_{L^{2}(H)}\lesssim{}1.$$
\end{cor}

\begin{proof}
Since
$$\norm{(-\Delta_{g}-\lambda^{2})u}_{L^{2}(M)}\lesssim{}\lambda$$
when we rescale with $\lambda^{-1}=h$
$$\norm{(h^{2}\Delta_{g}-1)u}_{L^{2}(M)}\lesssim{}h$$
and so $u$ is an $O_{L^{2}}(h)$ quasimode.  Suppose $\chi_{K}(r)$ is supported in $K\leq{}r\leq{}2K$ for some large $K$ then consider $\chi_{K}(h^{2}\Delta_{g}-1)u$. Since the principal symbol of $h^{2}\Delta_{g}-1$ is bounded below by $K$ on the support of $\chi_{K}$ we can invert the operator. Now since
$$\norm{(h^{2}\Delta-1)\chi_{K}(x,hD)u}_{L^{2}(M)}\lesssim{}h\norm{u}_{L^{2}(M)}$$
inverting $(h^{2}\Delta-1)$ gives us
$$\norm{\chi_{K}(x,hD)u}_{L^{2}(M)}\lesssim{}hK^{-1}\norm{u}_{L^{2}(M)}.$$
By standard semiclassical Sobolev estimates
$$\norm{hD_{x_{1}}\chi_{K}(h^{2}\Delta_{g}-1)u}_{L^{2}(H)}\lesssim{}h^{-1/2}K^{3/4}\norm{\chi_{K}(x,hD)u}_{L^{2}(M)}\lesssim{}h^{1/2}K^{-1/4}\norm{u}_{L^{2}(M)}.$$
To treat the region where $|\xi|_{g}>2$ we may then dyadically decompose to regions $2^{k+1}\leq{}|\xi|^{2}_{g}-1\leq{}2^{k+2}$ and sum to obtain
$$\norm{hD_{x_{1}}(1-\chi(x,hD))u}_{L^{2}(H)}\lesssim{}h^{1/2}\norm{u}_{L^{2}(M)}$$
which is considerably better than we want. Therefore we may focus on the region where $|\xi|_{g}\leq{}2$.  Let $\chi(x,\xi)$ be supported in $|\xi|_{g}\leq{}2$. 
Since $u$ is a quasimode and
$$(h^{2}\Delta-1)\chi(x,hD)u=\chi(x,hD)(h^{2}\Delta-1)u+O_{L^{2}}(h),$$
the function $\chi(x,hD)u$ is also a quasimode.  Working in Fermi normal coordinates in a small tubular neighbourhood of the hypersurface we may write
$$-h^{2}\Delta_{g}-1=p(x,hD)$$ where $p(x,hD)$ has principal symbol
$$p(x,\xi)=\xi_{1}^{2}+q(x,\xi').$$
Therefore 
$$\nu(x,\xi)=2\xi_{1}$$
$$\nu(x,hD)=2hD_{\nu}$$ and so by Theorem \ref{semiclassicalthm} with $h=\lambda^{-1}$
 $$\norm{\lambda^{-1}\partial_{\nu}u}_{L^{2}(H)}\lesssim{}1$$
 as required.

 \end{proof}
 
 This paper is organised in the following fashion. In Section \ref{sec:levelsets} we set out the basic semiclassical analysis used in this paper and prove an estimate on the $L^{2}$ mass of a quasimode concentrated in a $h$ dependent region of a level set $q(x,\xi)=K$. In Section \ref{sec:hypersurface} we specialise to the case where the level set is a hypersurface given by $x_{1}=0$ and prove Theorem \ref{semiclassicalthm}. Section \ref{sec:curved} uses the results of Sections \ref{sec:levelsets} and \ref{sec:hypersurface} to reproduce results on the restriction of eigenfunctions to curved hypersurfaces (the original results are due to Tataru \cite{tataru98} and Hu \cite{Hu} for Laplacians and Hassell-Tacy \cite{HTacy} for semiclassical operators). Section \ref{sec:sharp} provides some sharp examples to Theorem \ref{semiclassicalthm}.
 
 \subsection*{Acknowledgements}
 The author would like to thank Andrew Hassell for suggesting an investigation into the semiclassical result and for many helpful discussions. The author would like also to acknowledge the comments and suggestions of the reviewers which have greatly improved the paper.

 \section{Concentration localised near level sets}\label{sec:levelsets}

 Theorem \ref{semiclassicalthm} should be taken as a statement of non-concentration near the hypersurface $H$.  We can, using simple commutator relationships,  prove a weaker version that tells us about the $L^{2}$ mass concentrated in a $h^{\alpha}$ thickened neighbourhood of $H$. Actually we can make considerably more general statements about the concentration of an eigenfunction near a level set $q(x,\xi)=K$. To prove these estimates it is not necessary to assume any geometry of the level set, it does not even have to be a hypersurface. The key idea is that the concentration properties of $u$ should follow from the behaviour of the classical flow (via the classical-quantum correspondence principle). Particularly, that for trajectories of the classical flow to remain on a level set that level set must itself be invariant under the flow. That is if the classical flow if defined by
 $$\begin{cases}
 \dot{x}=\nabla_{\xi}p(x,\xi)\\
 \dot{\xi}=-\nabla_{x}p(x,\xi)\end{cases}$$
 then $q(x(t),\xi(t))$ must be independent of time. For general classical observable $q(x,\xi)$ we know that 
 $$\dot{q}(x,\xi)=\{q,p\}(x,\xi).$$
 The appropriate quantum analogue is then that a quasimode $u$ cannot concentrate in $L^{2}$ mass near a level set unless it is localised near a point where $\dot{q}(x,\xi)=0$. In this section we quantify this statement. For $u$ an $O_{L^{2}}(h)$ quasimode of $p(x,hD)$, we will prove estimates of the form
\begin{equation}\norm{\dot{q}(x,hD)u^{K,\alpha}}_{L^{2}}\lesssim{}h^{\alpha/2}\norm{u}_{L^{2}}\label{leveldecay}\end{equation}
where $u^{K,\alpha}$ is the component of $u$ localised in an $h^{\alpha}$ thickened region of $q(x,\xi)=K$. Estimate \eqref{leveldecay} tell us that where $\dot{q}(x,\xi)$ is large only a small amount of the $L^{2}$ mass of $u$ may be localised near the level set. On the other hand if $\dot{q}(x,\xi)\to{}0$ as $h\to{}0$ large concentrations can occur (in the extreme case all of the $L^{2}$ mass of $u$ may be localised within a $h^{\alpha}$ scale of $q(x,\xi)=K$. 
 
  In this section we work only with norms over the full manifold so to simplify notation we denote $L^{2}(M)$ by $L^{2}$. To state such results and throughout the rest of the paper we will need a number of cut off functions. Let $\chi^{i}:\R\to{}\R$ be a smooth function for $i=1,2,3$ defined by  

\begin{align*}
&  \chi^{1}(r)=\begin{cases}
 1 &|r|\leq 1\\
 0 &  |r|\geq 2\end{cases}\\
&  \chi^{2}(r)=\begin{cases}
 1 & 1\leq r\leq 2\\
 0 & r\leq{}1/2,r\geq{}5/2\end{cases}\\
&\chi^{3}(r)=\begin{cases}
 1 &r\geq 2\\
 0 &  r\leq 1.\end{cases}\end{align*}
  Let $q(x,\xi)\in{}S^{0}$ be a smooth symbol. Then for some fixed $K$ we denote 
  $$\chi^{i}_{\alpha,q}(x,\xi)=\chi^{i}(h^{-\alpha}(q(x,\xi)-K))$$ 
  and 
  $$\chi^{i}_{\alpha,q}(x,hD)u=Op(\chi^{i}_{\alpha,q}(x,\xi))u=\frac{1}{(2\pi{}h)^{n}}\iint{}e^{\frac{i}{h}\langle{}x-y,\xi\rangle}\chi_{\alpha,q}(x,\xi)u(y)d\xi dy.$$ 
  We then have the interpretation that for $i=1,2$ $\chi^{i}_{\alpha,q}(x,hD)u$ is the component of $u$ localised (at scale $h^{\alpha}$) near the set $q(x,\xi)=K$. For $i=3$ we have the interpretation that $\chi^{3}_{\alpha,q}(x,hD)$ localises $u$ to the region where $q(x,\xi)$ is positive with a $h^{\alpha}$ scale truncation. Since we assume that $u$ is semiclassically localised we can work in a compact subset of $T^{\star}M$ so we do not need to worry about defining decay of symbols as $|\xi|\to\infty$. However since we will be truncating on $h$ dependent scales we need to keep track of the loss in regularity of the symbol.
  
  \begin{defin}
  A symbol $q(x,\xi)$ is in the symbol class $h^{\beta}S^{m}$ if
  $$|D^{\gamma}_{x,\xi}q(x,\xi)|\leq{}C_{\gamma}h^{\beta-|\gamma|m}$$
  \end{defin}
  
  We will often need to compute the symbol of the compositions of two semiclassical pseudodifferential operators We use the standard expansion
  \begin{multline}
  p(x,hD)\circ{}q(x,hD)=Op(c(x,\xi))=\\
  c(x,\xi)=e^{ih\langle{}D_{\xi},D_{y}\rangle}p(x,\xi)q(y,\eta)\Big|_{x=y,\xi=\eta}\\
  =\sum_{k}\frac{h^{k}}{k!}\left(\frac{\langle{}D_{\xi},D_{y}\rangle}{i}\right)^{k}a(x,\xi)q(y,\eta)\Big|_{x=y,\xi=\eta.}\label{semiexp}\end{multline}
 We refer the reader to \cite{Zworski12} for details of the proof of this expansion (via stationary phase). The sum in \eqref{semiexp} should be taken as a semiclassical asymptotic sum. That is; if $p(x,\xi)\in{}h^{\beta_{1}}S^{m_{1}}$ and $q(x,\xi)\in{}h^{\beta_{2}}S^{m_{2}}$ for $m_{1}+m_{2}\leq{}1$ then for any $N$
 $$p(x,hD)\circ{}q(x,hD)=Op(c_{N}(x,\xi))+r_{N}(x,hD)$$
 where
 $$c_{N}(x,\xi)=\sum_{k=0}^{N}\frac{h^{k}}{k!}\left(\frac{\langle{}D_{\xi},D_{y}\rangle}{i}\right)^{k}a(x,\xi)q(y,\eta)\Big|_{x=y,\xi=\eta}$$
 and
 $$\norm{r_{N}(x,hD)}_{L^{2}\to{}L^{2}}\lesssim{}h^{\beta_{1}+\beta_{2}+N(1-m_{1}-m_{2})}.$$
 Note that if $m_{1}+m_{2}<1$ we may calculate the asymptotic  to any order (in $h$) error. If however $m_{1}+m_{2}=1$ we cannot get a full semiclassical expansion however we can often get some reasonable composition estimates.

To reduce the number of terms displayed  in any one expansion we adopt the following abuse of notation, that the exact value of a remainder symbol can change from line to line however the support and regularity properties remain the same. For instance we write
 $$p(x,hD)\chi^{1}_{\alpha,q}(x,hD)=\chi^{1}_{\alpha,q}(x,hD)p(x,hD)+h^{1-\alpha}\chi^{1}_{\alpha,q}(x,hD)r(x,hD)$$
and allow the symbol $r(x,\xi)$ to vary from line to line. 

We will on a number of occasions need to work with a regularised (on the scale $h^{\alpha}$) square root operator. 

 \begin{defin}\label{def:regsqrt}
 For $\alpha<1/2$ we define the regularised at scale $\alpha$ positive (and negative) square root $q^{1/2}_{\alpha,+}(x,hD)$ ($q^{1/2}_{\alpha,-}(x,hD)$) as 
 $$q_{\alpha,\pm}^{1/2}(x,hD)=Op(q^{1/2}_{\alpha,\pm}(x,\xi))$$
 where 
 $$q_{\alpha,+}^{1/2}(x,\xi)=h^{\frac{\alpha}{2}}(1-\chi^{3}(h^{-\alpha}q(x,\xi)))+(q(x,\xi))^{1/2}\chi^{3}(h^{-\alpha}q(x,\xi))$$
and
$$q_{\alpha,-}^{1/2}(x,\xi)=h^{\frac{\alpha}{2}}(1-\chi^{3}(-h^{-\alpha}q(x,\xi)))+(-q(x,\xi))^{1/2}\chi^{3}(-h^{-\alpha}q(x,\xi)).$$
By Lemma \ref{lem:invert} these operators are invertible. We denote their inverses by $q_{\alpha,\pm}^{-1/2}(x,hD).$
 \end{defin}
  
\begin{lem}\label{lem:invert}
Let $\alpha<1/2$. Suppose the regularised positive and negative powers of a symbol are given by
 $$q_{\alpha,+}^{\beta}(x,\xi)=h^{\alpha\beta}(1-\chi^{3}(h^{-\alpha}q(x,\xi)))+(q(x,\xi))^{\beta}\chi^{3}(h^{-\alpha}q(x,\xi))$$
and
$$q_{\alpha,-}^{\beta}(x,\xi)=h^{\alpha\beta}(1-\chi^{3}(-h^{-\alpha}q(x,\xi)))+(-q(x,\xi))^{\beta}\chi^{3}(-h^{-\alpha}q(x,\xi)).$$
Then both $q_{\alpha,\pm}^{\beta}(x,hD)$ are invertible and the symbols of their inverses are in $h^{-\beta}S^{\alpha}$.
\end{lem}

\begin{proof}
We will construct an inverse in the standard fashion using the semiclassical expansion \eqref{semiexp} for composition of operators. We will work only with $q^{\beta}_{\alpha,+}(x,hD)$, the proof for $q^{\beta}_{\alpha,-}(x,hD)$ is the same. First note that any derivatives applied to $q^{\beta}_{\alpha,+}(x,\xi)$ may fall on either the cut off function or on $q^{\beta}(x,\xi)$ (localised where $|q(x,\xi)|>h^{\alpha}$). Therefore any derivative can cause at most a loss of $h^{-\alpha}$ and so $q^{\beta}_{\alpha,+}(x,\xi)$ is at least in the symbol class $S^{\alpha}$. However we will need better estimates. Any derivative $D^{\gamma}$ of $q^{\beta}_{\alpha,+}(x,\xi)$ is made up of terms of the form
\begin{equation}h^{\alpha\beta}D^{\gamma_{0}}(1-\chi^{3}(h^{-\alpha}q(x,\xi)))\label{term1}\end{equation}
and
\begin{equation}q^{\beta-\gamma_{1}}\left[D^{\gamma_{1}}q(x,\xi)\right]D^{\gamma_{2}}\chi^{3}(h^{-\alpha}q(x,\xi))\label{term2}\end{equation}
with both $|\gamma_{0}|,(|\gamma_{1}|+|\gamma_{2}|)\leq{}|\gamma|$. Any terms of the form \eqref{term1} are localised to the region $q(x,\xi)\approx{}h^{\alpha}$ due to the support properties of $1-\chi^{3}$. Therefore
\begin{equation}|h^{\alpha\beta}D^{\gamma_{0}}(1-\chi^{3}(h^{-\alpha}q(x,\xi)))|\leq{}C_{\gamma}h^{\alpha\beta-|\gamma_{0}|\alpha}.\label{deriv1}\end{equation}
The size of the terms that come from \eqref{term2} depend on the size of $q(x,\xi)$. When $|q(x,\xi)|\geq{}4{}h^{\alpha}$ derivatives falling on the cut off function $\chi^{3}(h^{-\alpha}q(x,\xi))$ are zero. In this case we have
\begin{equation}\left|q^{\beta-|\gamma_{1}|}\left[D^{\gamma_{1}}q(x,\xi)\right]D^{\gamma_{2}}\chi^{3}(h^{-\alpha}q(x,\xi))\right|\leq{}\frac{C_{\gamma}}{|q(x,\xi)|^{\beta-|\gamma|}}.\label{deriv2}\end{equation}
On the other hand where $|q(x,\xi)|\leq{}4h^{\alpha}$ we obtain
\begin{equation}\left|q^{\beta-|\gamma_{1}|}\left[D^{\gamma_{1}}q(x,\xi)\right]D^{\gamma_{2}}\chi^{3}(h^{-\alpha}q(x,\xi))\right|\leq{}C_{\gamma}h^{\beta-\alpha|\gamma|}.\label{deriv3}\end{equation}
 Putting \eqref{deriv1}, \eqref{deriv2} and \eqref{deriv3} together we obtain
 \begin{equation}\left|D^{\gamma}q^{1/2}_{\alpha,+}(x,\xi)\right|\leq{}C_{\gamma}(h^{\alpha}+|q(x,\xi)|)^{\beta-|\gamma|}.\label{deriv4}\end{equation}
 Now we can construct an inverse. Let
$$b_{0}(x,\xi)=\frac{1}{q^{\beta}_{\alpha,+}(x,hD)}.$$
We want to compose $q^{\beta}_{\alpha,+}(x,hD)$ with $b_{0}(x,hD)$ and use \eqref{semiexp} to obtain a series expression for the symbol. So we need estimates on the derivatives of $b_{0}(x,\xi)$. Now $D^{\gamma}b_{0}$ is made up of terms of the form
$$\frac{1}{(q^{\beta}_{\alpha,+}(x,\xi))^{1+|\gamma_{1}|}}\left[Dq^{1/2}_{\alpha,+}(x,\xi)\right]^{|\gamma_{1}|-|\gamma_{2}|}D^{\gamma_{2}}\left(Dq^{1/2}_{\alpha,+}(x,\xi)\right)^{|\gamma_{2}|}$$
where $|\gamma_{1}|+|\gamma_{2}|=|\gamma|$. So from \eqref{deriv4} we obtain
$$\left|D^{\gamma}b_{0}(x,\xi)\right|\leq{}C_{\gamma}(h^{\alpha}+|q(x,\xi)|)^{-\beta-|\gamma|}.$$
 Clearly therefore $q^{\beta}_{\alpha,+}\in{}S^{\alpha}$ and $b_{0}\in{}h^{-\beta}S^{\alpha}$ so we can use \eqref{semiexp} to determine the symbol of the composition $q^{\beta}_{\alpha,+}(x,hD)\circ{}b_{0}(x,hD)$. From \eqref{semiexp} we have
 $$q^{\beta}_{\alpha,+}(x,hD)\circ{}b_{0}(x,hD)=Op(c_{0}(x,\xi))$$
 where
 $$c_{0}(x,\xi)=1+r_{0}(x,\xi)$$
and
 $$\left|D^{\gamma}r_{0}(x,\xi)\right|\leq{}C_{\gamma}h(h^{\alpha}+|q|)^{\beta-1-|\gamma|}(h^{\alpha}+|q|)^{-\beta-1-|\gamma|}\leq{}C_{\gamma}h(h^{\alpha}+|q(x,\xi)|)^{-2-2|\gamma|}.$$
 That is, the error from composing the principal symbols at least $h^{1-2\alpha}$ better than the identity. We may then define
 $$b_{1}(x,\xi)=\frac{r_{0}(x,\xi)}{q^{\beta}_{\alpha,+}(x,\xi)}$$
 so that
 $$q^{\beta}_{\alpha,+}(x,hD)\circ{}(b_{0}(x,hD)+b_{1}(x,hD))=Op(c_{1}(x,\xi))$$
 $$c_{1}(x,\xi)=1+r_{0}(x,\xi)-r_{0}(x,\xi)+r_{1}(x,\xi)$$
 where 
 $$\left|D^{\gamma}r_{1}(x,\xi)\right|\leq{}C_{\gamma}h^{2}(h^{\alpha}+|q(x,\xi)|)^{-4-2|\gamma|}.$$
 Then in a similar fashion to the standard case we keep recursively defining $b_{j}(x,\xi)=\frac{r_{j-1}(x,\xi)}{q^{1/2}_{\alpha,+}(x,\xi)}$ and set
 $$b_{N}(x,\xi)=\sum_{j=0}^{N}b_{j}(x,\xi).$$
 Since we gain a factor of $h^{1-2\alpha}$ with each successive term we can find a $O(h^{\infty})$ inverse. \end{proof}
 
 On a number of occasions through this paper we will need to ``commute'' these regularised square roots through cut off functions in $x_{1}$ down to the scale dictated by the uncertainty principle. 
 
 \begin{lem}\label{lem:comm}
 Let $\alpha<1/2$, $\beta\leq{}1-\alpha$  and $\chi:\R\to\R$ smooth. Then
 $$q_{\alpha,\pm}^{1/2}(x,hD)\circ{}\chi^{1}(h^{-\beta}x_{1})\circ{}q_{\alpha,\pm}^{-1/2}(x,hD)=c(x,hD)$$
 where $c(x,hD):L^{2}\to{}L^{2}$ with bound uniform in $h$  and the symbol $c(x,\xi)$ obeys
 $$|D^{N}_{x_{1}}c(x,\xi)|\leq{}C_{N}h^{-\beta N}$$
 $$|D^{\gamma}_{x',\xi}c(x,\xi)|\leq{}C_{\gamma}h^{-\alpha|\gamma|}.$$
If further $\chi$ is compactly supported, for any $M>0$
 $$|D^{N}_{x_{1}}c(x,\xi)|\leq{}C_{N}h^{-\beta N}(1+h^{-\beta}|x|)^{-M}$$
 $$|D^{\gamma}_{x',\xi}c(x,\xi)|\leq{}C_{\gamma}h^{-\alpha|\gamma|}(1+h^{-\beta}|x|)^{-M}.$$
 The same is true of $q_{\alpha,\pm}^{-1/2}(x,hD)\circ{}\chi^{1}(h^{-\beta}x_{1})q_{\alpha,\pm}^{1/2}(x,hD)$. 
 \end{lem}
 
 \begin{proof}
 Note that if $\beta<1-\alpha$ we can simply use the asymptotic expansion \eqref{semiexp} and the estimates of Lemma \ref{lem:invert}. However we do need to carry down to the scale $\beta=1-\alpha$ so we assume we are at that scale. We write
\begin{align*}q_{\alpha,\pm}^{1/2}&(x,hD)\circ{}\chi^{1}(h^{-(1-\alpha)}x_{1})\circ{}q_{\alpha,\pm}^{-1/2}(x,hD)u\\
&=\frac{1}{(2\pi{}h)^{2n}}\int{}e^{\frac{i}{h}\left(\langle{}x-y,\xi\rangle+\langle{}y-z,\eta\rangle\right)}q_{\alpha,\pm}^{1/2}(x,\xi)\chi^{1}(h^{-(1-\alpha)}y_{1})q^{-1/2}_{\alpha,\pm}(y,\eta)u(z)dyd\xi d\eta dz\\
 &=\frac{1}{(2\pi{}h)^{n}}\int{}e^{\frac{i}{h}\langle{}x-z,\eta\rangle}c(x,\eta)u(z)d\eta dz\end{align*}
 where
 $$c(x,\xi)=\frac{1}{(2\pi{}h)^{n}}\int{}e^{\frac{i}{h}\left(\langle{}x,\xi-\eta\rangle-\langle{}y,\xi-\eta\rangle\right)}q_{\alpha,\pm}^{1/2}(x,\xi)\chi^{1}(h^{-(1-\alpha)}y_{1})q^{-1/2}_{\alpha,\pm}(y,\eta)dyd\xi.$$
 Now $D^{N}_{x_{1}}c(x,\xi)$ consists of terms of the form
 $$\frac{1}{(2\pi{}h)^{n}}\int{}e^{\frac{i}{h}\left(\langle{}x,\xi-\eta\rangle-\langle{}y,\xi-\eta\rangle\right)}\left(\frac{\xi_{1}-\eta_{1}}{h}\right)^{N_{1}}D^{N_{2}}_{x_{1}}q_{\alpha,\pm}^{1/2}(x,\xi)\chi^{1}(h^{-(1-\alpha)}y_{1})q^{-1/2}_{\alpha,\pm}(y,\eta)dyd\xi$$
 where $N_{1}+N_{2}=N$. Therefore by integrating by parts in $y_{1}$ this becomes
 \begin{align*}\frac{1}{(2\pi{}h)^{n}}\int{}&e^{\frac{i}{h}\left(\langle{}x,\xi-\eta\rangle-\langle{}y,\xi-\eta\rangle\right)}D^{N_{2}}_{x_{1}}q_{\alpha,\pm}^{1/2}(x,\xi)D^{N_{1}}_{y_{1}}\left[\chi^{1}(h^{-(1-\alpha)}y_{1})q^{-1/2}_{\alpha,\pm}(y,\eta)\right]dyd\xi\\
& \frac{1}{(2\pi{}h)^{n}}\int{}e^{\frac{i}{h}\left(\langle{}x,\xi-\eta\rangle-\langle{}y,\xi-\eta\rangle\right)}b_{\alpha}^{1}(x,\xi)d_{\alpha}^{1}(y,\eta)dyd\xi\end{align*}
 where we inherit the estimates of Lemma \ref{lem:invert}
 $$|D^{\gamma}b_{\alpha}^{1}(x,\xi)|\leq{}C_{\gamma,N}(h^{\alpha}+|q(x,\xi)|)^{1/2-|\gamma|-N_{2}}$$
 $$|D^{\gamma}_{y',\eta}d_{\alpha}^{1}(y,\eta)|\leq{}C_{\gamma,N}h^{-N_{1}(1-\alpha)}(h^{\alpha}+|q(y,\eta)|)^{-1/2-|\gamma|}$$
 $$|D^{J}_{y_{1}}d_{\alpha}^{1}(y,\eta)|\leq{}C_{N}h^{-(1-\alpha)(N_{1}+J)}(h^{\alpha}+|q(y,\eta)|)^{-1/2}.$$
 Note that if $\chi$ has compact support $d_{\alpha}^{1}(y,\eta)$ is supported where $|y|\leq{}2h^{1-\alpha}$. There is a nondegenerate critical point in the phase function,
 $$\langle x,\xi-\eta\rangle-\langle y,\xi-\eta\rangle,$$
 at the point $\xi=\eta$, $x=y$. Consider
 $$I^{N}_{k_{1},k_{2}}=\frac{1}{(2\pi{}h)^{n}}\int{}e^{\frac{i}{h}\left(\langle{}x,\xi-\eta\rangle-\langle{}y,\xi-\eta\rangle\right)}b_{\alpha}^{1}(x,\xi)d_{\alpha}^{1}(y,\eta)\chi^{2}(2^{-k_{1}}h^{-\alpha}|\xi-\eta|)\chi^{2}(2^{-k_{2}}h^{-(1-\alpha)}|x-y|)dyd\xi$$
 Integration by parts in $y$ and $\xi$ and using the fact that
 $$|d^{1}_{\alpha}(x,\xi)-d^{1}_{\alpha}(y,\eta)|\leq{}(2^{k_{1}}+2^{k_{2}})$$
 gives
 $$|I^{N}_{k_{1},k_{2}}|\leq{}h^{-(1-\alpha)N}2^{-k_{1}-k_{2}}.$$
 Therefore we may dyadically sum to obtain
 $$|D^{N}_{x_{1}}c(x,\xi)|\leq{}C_{N}h^{-\beta N}.$$
 Further, if $\chi$ is compactly supported there is no critical point within $h^{1-\alpha}$ of $x_{1}$ when $|x_{1}|\gg{}h^{1-\alpha}$. Therefore in that case we obtain 
 $$|D^{N}_{x_{1}}c(x,\xi)|\leq{}C_{N}h^{-\beta N}(1+h^{-\beta}|x|)^{-M}.$$
 The derivatives  $D_{x'}c(x,\eta)$ have the same form however since $\chi^{1}(h^{-(1-\alpha)}y_{1})$ does not depend on $y'$ we get a better bound of  
\begin{align*}&|D^{\gamma}_{x'}c(x,\eta)|\leq{}C_{\gamma}h^{-\alpha|\gamma|},\\
 &|D^{\gamma}_{x'}c(x,\eta)|\leq{}C_{\gamma}h^{-\alpha|\gamma|}(1+h^{-\beta}|x|)^{-M}\quad\chi\text{ compactly supported}.\end{align*}
The derivatives $D^{N}_{\eta_{i}}c(x,\eta)$ are given by terms of the form
$$\frac{1}{(2\pi{}h)^{n}}\int{}e^{\frac{i}{h}\left(\langle{}x,\xi-\eta\rangle-\langle{}y,\xi-\eta\rangle\right)}\left(\frac{x_{i}-y_{i}}{h}\right)^{N_{1}}q_{\alpha,\pm}^{1/2}(x,\xi)\chi^{1}(h^{-\beta}y_{1})D^{N_{2}}_{\eta_{i}}q^{-1/2}_{\alpha,\pm}(y,\eta)dyd\xi$$
and so integration by parts in $\xi_{i}$ followed by the same argument gives
\begin{align*}&|D^{\gamma}_{\eta}c(x,\eta)|\leq{}C_{\gamma}h^{-\alpha|\gamma|},\\
&|D^{\gamma}_{\eta}c(x,\eta)|\leq{}C_{\gamma}h^{-\alpha|\gamma|}(1+h^{-\beta}|x|)^{-M}\quad\chi\text{ compactly supported}.\end{align*}
Therefore by the rescaling and almost orthogonality arguments of \cite{Zworski12} we can see that $c(x,hD)$ maps $L^{2}\to{}L^{2}$ with uniformly bounded norm.
 \end{proof}
  
 We now prove a concentration theorem for $h^{\alpha}$ thickened neighbourhoods of level sets. This theorem applies not only for the level set $x_{1}=0$ but  for level sets of all smooth symbols $q(x,\xi)$. 
 \begin{thm}\label{commutatorthm}
 Suppose $u,v\in{}L^{2}$.  Let $q(x,\xi)$ be a smooth symbol and define 
 $$\dot{q}(x,\xi)=\{p(x,\xi),q(x,\xi)\}.$$
 Then for $\alpha\leq{}\frac{1}{2}$
  \begin{multline}\left|\langle{}\chi_{\alpha,q}(x,hD)v,\dot{q}(x,hD)\chi_{\alpha,q}(x,hD)u\rangle\right|\lesssim{}h^{\alpha-1}(\norm{E[v]}_{L^{2}}\norm{u}_{L^{2}}+\norm{v}_{L^{2}}\norm{E[u]}_{L^{2}})\\
  +h^{1-\alpha}\norm{u}_{L^{2}}\norm{v}_{L^{2}}\quad{}i=1,2\label{innerthmest}\end{multline}
   and
\begin{equation}\norm{\dot{q}(x,hD)\chi^{i}_{\alpha,q}(x,hD)u}_{L^{2}}\lesssim{}h^{\frac{\alpha}{2}-\frac{1}{2}}(\norm{E[u]}_{L^{2}}\norm{u}_{L^{2}})^{\frac{1}{2}}+h^{\frac{1}{2}-\frac{\alpha}{2}}\norm{u}_{L^{2}}\quad{}i=1,2.\label{normthmest}\end{equation}
In particular, if both $u$ and $v$ are $O_{L^{2}}(h)$ quasimodes, then 
$$\left|\langle{}\chi_{\alpha,q}(x,hD)v,\dot{q}(x,hD)\chi_{\alpha,q}(x,hD)u\rangle\right|\lesssim{}h^{\alpha}\norm{u}_{L^{2}}\norm{v}_{L^{2}}\quad{}i=1,2$$
and
$$\norm{\dot{q}(x,hD)\chi^{i}_{\alpha,q}(x,hD)u}_{L^{2}}\lesssim{}h^{\frac{\alpha}{2}}\norm{u}_{L^{2}}\quad{}i=1,2.$$

 \end{thm}
 
 \begin{remark}
 Theorem \ref{commutatorthm} tells us that if $|\dot{q}(x,\xi)|>c>0$, 
 $$\norm{\chi^{i}_{\alpha,q}(x,hD)u}_{L^{2}}\lesssim{}h^{\alpha/2}\norm{u}_{L^{2}}\quad{}i=1,2.$$ 
 That is there cannot be concentration near $q(x,\xi)=K$. The notation $\dot{q}(x,\xi)$ for the Poisson bracket comes from the fact that this gives the classical evolution of $q(x(t),\xi(t))$.\end{remark}
 
 \begin{remark}
 If we set $q(x,\xi)=x$ and $K=0$ Theorem \ref{commutatorthm} tells us that
 $$\norm{\nu(x,hD)\chi^{i}(h^{-\alpha}x_{1})u}_{L^{2}}\lesssim{}h^{\alpha/2}\norm{u}_{L^{2}}\quad{}i=1,2.$$
 That is that $\nu(x,hD)u$ is small in a $h^{\alpha}$ thickened neighbourhood of the hypersurface $x_{1}=0$. 
 \end{remark}
 
 \begin{proof}
 Let $\zeta(r):\R\to{}\R$ be smooth and compactly supported. Set $\tlide{\zeta}(r)$ such that $\tilde{\zeta}'(r)=\zeta(r)$. Then denote
 $$\tilde{\zeta}_{\alpha,q}(x,\xi)=\tilde{\zeta}(h^{-\alpha}(q(x,\xi)-K)).$$
 Now consider the commutator $[p(x,hD),\tilde{\zeta}_{\alpha,q}(x,hD)]$.
 We know that the principal symbol of the commutator is given by the Poisson bracket. That is
 \begin{equation}[p(x,hD),\tilde{\zeta}_{\alpha,q}(x,hD)]=Op(c(x,\xi))\label{commutatorID}\end{equation}
 \begin{multline}c(x,\xi)=ih\{p(x,\xi),\tilde{\zeta}_{\alpha,q}(x,\xi)\}+O(h^{2(1-\alpha)})\\
 =ih^{1-\alpha}\zeta(h^{-\alpha}(q(x,\xi)-K))\{p(x,\xi),q(x,\xi)\}+O(h^{2(1-\alpha)})\\
 =ih^{1-\alpha}\dot{q}(x,\xi)\zeta(h^{-\alpha}(q(x,\xi)-K))+O(h^{2(1-\alpha)}).\label{commutatorsymb}\end{multline}
 Rearranging \eqref{commutatorID} in view of \eqref{commutatorsymb} we obtain
$$\dot{q}(x,hD)\zeta_{\alpha,q}(x,hD)u=\frac{h^{\alpha-1}}{i}[p(x,hD),\tilde{\zeta}_{\alpha,q}(x,hD)]u+O(h^{1-\alpha}\norm{u}_{L^{2}}).$$
Now consider the inner product
 \begin{equation}
 \begin{aligned}
 |\langle{}v,\dot{q}(x,hD)&\zeta_{\alpha,q}(x,hD)u\rangle|\\
& \leq{}h^{\alpha-1}\left|\langle{}v,[p(x,hD),\tilde{\zeta}_{\alpha,q}(x,hD)]u\rangle\right|+O(h^{1-\alpha})\norm{u}_{L^{2}}\norm{v}_{L^{2}}\\
 &\leq{}h^{\alpha-1}\left(\left|\langle{}p^{\star}(x,hD)v,\tilde{\zeta}_{\alpha,q}(x,hD)u\rangle\right|+\left|\langle{}v,\tilde{\zeta}_{\alpha,q}(x,hD)p(x,hD)u\rangle\right|\right)\\
 &\quad\quad\quad+O(h^{1-\alpha})\norm{u}_{L^{2}}\norm{v}_{L^{2}} \\
&  \lesssim{}h^{\alpha-1}\left(\norm{E[v]}_{L^{2}}\norm{u}_{L^{2}}+\norm{v}_{L^{2}}\norm{E[u]}_{L^{2}}\right)+h^{1-\alpha}\norm{v}_{L^{2}}\norm{u}_{L^{2}}.\label{innerprodest}\end{aligned}\end{equation}
If we set $\zeta(r)=(\chi^{i})^{2}(r)$ for $i=1,2$ we obtain \eqref{innerthmest}. Finally by setting $\zeta(r)=(\chi^{i})^{2}(r)$ and $v=\dot{q}(x,hD)u$ we obtain
\begin{multline*}|\langle\dot{q}(x,hD)\chi_{\alpha,q}(x,hD)u,\dot{q}(x,hD)\chi_{\alpha,q}(x,hD)u\rangle|\\
=\left|\langle{}\dot{q}(x,hD)u,\dot{q}(x,hD)\zeta_{\alpha,q}(x,hD)u\rangle\right|+O(h^{1-\alpha}\norm{u}_{L^2}^{2})\\
\lesssim{}h^{1-\alpha}\norm{E[u]}_{L^{2}}\norm{u}_{L^{2}}+O(h^{1-\alpha}\norm{u}_{L^{2}}^{2})\end{multline*}
which yields \eqref{normthmest}.
\end{proof}

While not enough to obtain restriction estimates directly, Theorem \ref{commutatorthm} will be very useful to us. We begin by looking at some immediate corollaries. 

Note that the inner product version of Theorem \ref{commutatorthm} is stronger that the norm version. We can in fact easily show that  $\dot{q}_{\alpha,\pm}^{1/2}(x,hD)$ obeys
$$\norm{\dot{q}^{1/2}(x,hD)u}_{L^{2}}\lesssim{}h^{\frac{\alpha}{2}}\norm{u}_{L^{2}}$$
where $u$ is an $O_{L^{2}}(h)$ quasimode of $p(x,hD)$. 

\begin{cor}\label{cor:sqrt}
Let $q(x,\xi),p(x,\xi)$ and $\dot{q}(x,\xi)$ be as in Theorem \ref{commutatorthm} and suppose $u\in{}L^{2}$. Then for $\alpha<1/2$
$$\norm{\dot{q}_{\alpha,\pm}^{1/2}(x,hD)\chi^{i}_{\alpha,q}(x,hD)u}_{L^{2}}\lesssim{}h^{\frac{\alpha}{2}-\frac{1}{2}}(\norm{E[u]}_{L^{2}}\norm{u}_{L^{2}})^{\frac{1}{2}}+h^{\frac{\alpha}{2}}\norm{u}_{L^{2}}\quad{}i=1,2.$$
In particular if $u$ is an $O_{L^{2}}(h)$ quasimode of $p(x,hD)$ then
$$\norm{\dot{q}_{\alpha,\pm}^{1/2}(x,hD)\chi^{i}_{\alpha,q}(x,hD)u}_{L^{2}}\lesssim{}h^{\frac{\alpha}{2}}\norm{u}_{L^{2}}\quad{}i=1,2.$$
\end{cor}

\begin{proof}
We write
\begin{align*}
\langle{}\dot{q}_{\alpha,+}^{1/2}(x,hD)\chi^{i}_{\alpha,q}(x,hD)u,&\dot{q}_{\alpha,+}^{1/2}(x,hD)\chi^{i}_{\alpha,q}(x,hD)u\rangle\\
&=\langle\chi^{i}_{\alpha,q}(x,hD)u,\dot{q}_{\alpha,+}^{1/2}(x,hD)\circ\dot{q}_{\alpha,+}^{1/2}(x,hD)\chi^{i}_{\alpha,q}(x,hD)u\rangle\\
&+O(h^{1-\alpha}\norm{u}_{L^{2}}^{2}).\end{align*}
Now applying \eqref{semiexp} and the bound from Lemma \ref{lem:invert} we have
$$\dot{q}_{\alpha,+}^{1/2}(x,hD)\circ\dot{q}_{\alpha,+}^{1/2}(x,hD)=Op(c(x,\xi))$$
where
$$c(x,\xi)=(\dot{q}_{\alpha,+}^{1/2}(x,\xi))^{2}+hr(x,\xi)$$
$$|D^{\gamma}r(x,\xi)|\leq{}C_{\gamma}(h^{\alpha}+|\dot{q}(x,\xi)|)^{-1+|\gamma|}.$$
Note that $(\dot{q}_{\alpha,+}^{1/2}(x,\xi))^{2}$ differs from $\dot{q}(x,\xi)$ only in the region where $|\dot{q}(x,\xi)|\leq{}2h^{\alpha}$ so
\begin{align*}\langle\chi^{i}_{\alpha,q}(x,hD)&u,\dot{q}_{\alpha,+}^{1/2}(x,hD)\circ\dot{q}_{\alpha,+}^{1/2}(x,hD)\chi^{i}_{\alpha,q}(x,hD)u\rangle\\
&\langle{}\langle\chi^{i}_{\alpha,q}(x,hD)u,\dot{q}(x,hD)\chi^{i}_{\alpha,q}(x,hD)u\rangle+h^{\alpha}\norm{u}_{L^{2}}^{2}.\end{align*}
Now applying Theorem \ref{commutatorthm} we obtain
\begin{align*}\Big|\langle{}\dot{q}_{\alpha,+}^{1/2}(x,hD)\chi^{i}_{\alpha,q}(x,hD)u,\dot{q}_{\alpha,+}&^{1/2}(x,hD)\chi^{i}_{\alpha,q}(x,hD)u\rangle\Big|\\
&\lesssim{}h^{\alpha-1}\norm{E[u]}_{L^{2}}\norm{u}_{L^{2}}+h^{\alpha}\norm{u}_{L^{2}}^{2}\end{align*}
which completes the proof. The proof for $\dot{q}_{\alpha,-}^{1/2}(x,hD)$ is the same, so we omit it. 
\end{proof}

Another interesting consequence of Theorem \ref{commutatorthm} is that applying these kind of cut off functions do not damage the quasimode order  as much as may be first sumised. Indeed if $\zeta(x,hD)$ is a semiclassical psuedodifferential operator that localised $u$ at in $|\nu(x,\xi)|$ at scale $h^{\alpha}$ then the symbol of $\zeta(x,hD)$ must be in $S^{\alpha}$. Then
$$p(x,hD)\zeta(x,hD)u=\zeta(x,hD)p(x,hD)u+O_{L^{2}}(h^{1-\alpha}\norm{u}_{L^{2}}).$$
So if $\alpha\gg{}0$ placing cut offs on quasimodes appears to damage their quasimode error quite significantly. However Theorem \ref{commutatorthm} allows us to somewhat correct the error term. 

\begin{cor}\label{cuttofferror}
Suppose $\alpha\leq{}\frac{1}{2}$ and $\chi^{i}_{\alpha,q}(x,hD)$ is as in Theorem \ref{commutatorthm}. Then if $u$ and $v$ are $O_{L^{2}}(h)$ quasimodes of $p(x,hD)$,
\begin{equation}\left|\langle{}\chi^{i}_{\alpha,q}(x,hD)v,E[\chi_{\alpha,q}(x,hD)u]\rangle\right|\lesssim{}h\norm{v}_{L^{2}}\norm{u}_{L^{2}}\quad{}i=1,2,3\label{inner-err}\end{equation}
and
\begin{equation}\norm{E[\chi_{\alpha,q}(x,hD)u}_{L^{2}}\lesssim{}h^{1-\frac{\alpha}{2}}\norm{u}_{L^{2}}\quad{}i=1,2,3.\label{norm-err}\end{equation}
\end{cor}

\begin{proof}
We know that the principal symbol of the commutator is given by the Poisson bracket so (in the notation of Theorem \ref{commutatorthm})
\begin{multline}p(x,hD)\chi^{i}_{\alpha,q}(x,hD)u=\chi^{i}_{\alpha,q}(x,hD)p(x,hD)u\\+h^{1-\alpha}(\chi^{i})'_{\alpha,q}(x,hD)\dot{q}(x,hD)u
+O_{L^{2}}(h^{2(1-\alpha)}\norm{u}_{L^{2}}).\label{applyp}\end{multline}
The first and the third terms are already $O_{L^{2}}(h)$ so we need only to treat the middle term. That is we need to estimate
$$h^{1-\alpha}\langle{}\chi^{i}_{\alpha,q}v,(\chi^{i})'_{\alpha,q}(x,hD)\dot{q}(x,hD)u\rangle=h^{1-\alpha}\langle{}v,\chi^{i}_{\alpha,q}(\chi^{i})'_{\alpha,q}\dot{q}(x,hD)u\rangle+O(h\norm{u}_{L^{2}}\norm{v}_{L^{2}}).$$
We apply the proof of Theorem \ref{commutatorthm} (in particular the inequality \eqref{innerprodest}) with $\zeta(r)=\chi^{i}(r)(\chi^{i})'(r)$ to obtain
$$\left|\langle{}\chi^{i}_{\alpha,q}v,E[\chi_{\alpha,q}(x,hD)u]\rangle\right|\lesssim{}h\norm{v}_{L^{2}}\norm{u}_{L^{2}}.$$
To get \eqref{norm-err} we again only have to treat the middle term of \eqref{applyp}. Theorem \ref{commutatorthm} tells us that
$$\norm{\dot{q}(x,hD)\chi_{\alpha,\nu}(x,hD)u}_{L^{2}}\lesssim{}h^{\alpha/2}\norm{u}_{L^{2}}$$
which immediately implies \eqref{norm-err}.
\end{proof}

\section{Hypersurface concentration bounds}\label{sec:hypersurface}

We now address the more difficult question of hypersurface $L^{2}$ bounds. In this section we specialise to $q(x,\xi)=x_{1}$ and $K=0$. Then
$$\dot{q}(x,\xi)=\{p(x,\xi),x_{1}\}=\partial_{\xi_{1}}p(x,\xi)=\nu(x,\xi).$$ 
In what follows we adopt the convention that $\langle\cdot,\cdot\rangle_{x'}$ is the inner product on the hypersurface $x_{1}=0$. 

We will prove Theorem \ref{semiclassicalthm} by splitting the analysis into two parts. the tangential component, localised where $|\nu(x,\xi)|\leq{}h^{1/3}$, and the non-tangential component, where $|\nu(x,\xi)|\geq{}h^{1/3}$. Indeed for the non-tangential contribution we in fact prove the stronger statement that
 $$\norm{\nu(x,hD)\chi^{2}_{\alpha,\nu}(x,hD)u}_{L^{2}(H)}\lesssim{}h^{\alpha/2}\norm{u}_{L^{2}(M)}.$$
 In the tangential case we are able to prove the strong version where $p(x,hD)$ is sufficiently Laplace-like with respect to the hypersurface (see Definition \ref{laplacelike}). For general $p(x,hD)$ we are still however able to obtain the weaker statement
 $$\norm{\nu(x,hD)\chi^{1}_{\alpha,\nu}(x,hD)v}_{L^{2}(H)}\lesssim{}\norm{u}_{L^{2}(M)}$$
 which is enough to obtain Theorem \ref{semiclassicalthm}.

We produce an operator $W$ which has the effect of changing variables, but fixing the hypersurface, so that the pseudodifferential operator $p(x,hD)$ becomes the simple, constant coefficient differential operator $hD_{x_{1}}$.  
 
 \begin{prop}\label{prop:semiprop}
 There exists an operator $W:L^{2}(\R^{n})\to{}[0,\epsilon]\times{}L^{2}(\R^{n-1})$ such that
 $$hD_{x_{1}}\circ{}W=W\circ{}p(x,hD)+O(h^{\infty})$$
 and
 $$Wu\Big|_{H}=u\Big|_{H}.$$
 Further, $W$ is given by a  semiclassical Fourier integral operator
\begin{equation}Wu=\frac{1}{(2\pi{}h)^{n}}\iint{}e^{\frac{i}{h}(\langle{}x',\xi'\rangle+\phi(x_{1},y,\xi))}b(x_{1},y,\xi)u(y)d\xi{}dy\label{semipara}\end{equation}
with
 $$\begin{cases}\partial_{x_{1}}\phi+p(y,\nabla_{y}\phi)=0\\
 \phi(0,y,\xi)=-\langle{}y,\xi\rangle\\
b(0,y,\xi)=1.\end{cases}$$
 \end{prop}
 
 \begin{proof}
 This is just an adaption of a standard semiclassical parametrix (see for example \cite{Zworski12}).  
 If $W$ is given by \eqref{semipara} then 
 $$hD_{x_{1}}\circ{}Wu=\frac{1}{(2\pi{}h)^{n}}\iint{}e^{\frac{i}{h}(\langle{}x',\xi'\rangle+\phi(x_{1},y,\xi))}\left(\partial_{x_{1}}\phi(x_{1},y,\xi)b(x_{1},y,\xi)+hD_{x_{1}}b(x_{1},y,\xi)\right)u(y)d\xi{}dy.$$
 On the other hand 
 $$W\circ{}p(x,hD)u=\frac{1}{(2\pi{}h)^{2n}}\iint{}e^{\frac{i}{h}(\langle{}x',\xi'\rangle+\phi(x_{1},z,\xi)+\langle{}z-y,\eta\rangle)}b(x_{1},z,\xi)p(z,\eta)u(y)dyd\eta{}d\xi{}dz.$$
 We calculate the $(z,\eta)$ integral via the method of stationary phase. The phase is stationary when
 $$y=z\quad{}\nabla_{y}\phi(x_{1},y,\xi)=\eta$$
 and the critical point is clearly non-degenerate. So
$$W\circ{}p(x,hD)u=\frac{1}{(2\pi{}h)^{n}}\iint{}e^{\frac{i}{h}(\langle{}x',\xi'\rangle+\phi(x_{1},y,\xi))}\left(b(x_{1},y,\xi)p(x,\nabla_{y}\phi)+hr_{1}(x_{1},y,\xi)\right)u(y)d\xi{}dy.$$
Clearly if $\phi$ satisfies
  $$\partial_{x_{1}}\phi(x_{1},y,\xi)+p(y,\nabla_{y}\phi)=0$$
  we remove the highest order term. Note that this is just a Hamilton-Jacobi equation. We may then solve away lower terms in the standard fashion by expressing $b$ as a series. That is
$$b(x_{1},y,\xi)\sim\sum_{k}h^{k}b_{k}(x_{1},y,\xi)$$ 
with
$$\begin{cases}
b_{0}(0,y,\xi)=1\\
b_{k}(0,y,\xi)=0&k\geq{}1.\end{cases}$$
Finally we check the hypersurface condition. When $x_{1}=0$ we have
$$Wu=\frac{1}{(2\pi{}h)^{n}}\iint{}e^{\frac{i}{h}(\langle{}x'-y',\xi'\rangle-y_{1}\xi_{1})}u(y)d\xi{}dy=u\Big|_{H}$$
so
$$Wu\Big|_{H}=u\Big|_{H}.$$ 
 \end{proof}
 
 We now use this variable change to prove bounds on the inner product $\langle{}\chi^{i}_{\alpha,\nu}(x,hD)v,  \chi^{i}_{\alpha,\nu}(x,hD)u\rangle_{x'}$.
 
 \begin{prop}\label{prop:nontang}
Suppose $u,v\in{}L^{2}$, $i=2,3$ and $\alpha\leq{}1/3$ then
\begin{equation}
\begin{aligned}
\Big|\langle{}\chi^{i}_{\alpha,\nu}&(x,hD)v,  \chi^{i}_{\alpha,\nu}(x,hD)u\rangle_{x'}\Big|\\
&\lesssim{}h^{-\alpha}\norm{\nu^{-1/2}_{\alpha,\pm}(x,hD){\chi}^{i}_{\alpha,\nu}(x,hD)u}_{L^{2}(M)}\norm{\nu^{-1/2}_{\alpha,\pm}(x,hD){\chi}^{i}_{\alpha,\nu}(x,hD)v}_{L^{2}(M)}\\
&+h^{-1}\Big(\norm{\nu_{\alpha,+}^{-1/2}(x,hD)E[{\chi}^{i}_{\alpha,\nu}(x,hD)u]}_{L^{2}(M)}\norm{\nu^{-1/2}_{\alpha,+}(x,hD){\chi}^{i}_{\alpha,\nu}(x,hD)v}_{L^{2}(M)}\\
&+\norm{\nu^{-1/2}_{\alpha,\pm}(x,hD){\chi}^{i}_{\alpha,\nu}(x,hD)u}_{L^{2}(M)}\norm{\nu^{-1/2}_{\alpha,\pm}(x,hD)E[{\chi}^{i}_{\alpha,\nu}(x,hD)v]}_{L^{2}(M)}\Big)\label{hypersurfaceinner}\end{aligned}\end{equation}
where $\nu^{-1/2}_{\alpha,\pm}(x,hD)$ is the inverse of $\nu^{1/2}_{\alpha,\pm}(x,hD)$.
 \end{prop}
 
 \begin{proof}

Let $\theta(x_{1})$ be the Heaviside function
$$\theta(x_{1})=\begin{cases}
\frac{1}{2}&x_{1}>0\\
-\frac{1}{2}& x_{1}<0.\end{cases}$$
Then
$$hD_{x_{1}}\theta(x_{1})\chi^{1}(h^{-\alpha}x_{1})W=\frac{h}{i}\delta(x_{1})W+\theta(x_{1})\chi^{1}(h^{-\alpha}x_{1})hD_{x_{1}}W+\frac{h^{1-\alpha}}{i}\theta(x_{1})(\chi^{1})'(h^{-\alpha}x_{1})W$$
and rearranging  we obtain
$$\delta(x_{1})W=ih^{-1}\left(hD_{x_{1}}\theta(x_{1})\chi^{1}(h^{-\alpha}x_{1})W-\theta(x_{1})\chi^{1}(h^{-\alpha}x_1)hD_{x_{1}}W\right)+h^{-\alpha}\theta(x_{1})(\chi^{1})'(h^{-\alpha}x_{1})W$$
so if $f$ and $g$ are in $L^{2}$,
\begin{align*}\langle{}f,g\rangle_{x'}&=\langle{}Wf,Wg\rangle_{x'}\\
&=-ih^{-1}\Big(\langle{}hD_{x_{1}}Wf,\theta(x_{1})\chi^{1}(h^{-\alpha}x_{1})Wg-\langle{}Wf,\theta(x_{1})\chi^{1}(h^{-\alpha}x_{1})hD_{x_{1}}Wg\rangle\Big)\\
&\quad\quad\quad+h^{-\alpha}\langle{}Wf,\theta(x_{1})(\chi^{1})'(h^{-\alpha}x_{1})Wg\rangle\\
&=-ih^{-1}\Big(\langle{}WE[f],\theta(x_{1})\chi^{1}(h^{-\alpha}x_{1})Wg-\langle{}Wf,\theta(x_{1})\chi^{1}(h^{-\alpha}x_{1})WE[g]\rangle\Big)\\
&\quad\quad\quad+h^{-\alpha}\langle{}Wf,\theta(x_{1})(\chi^{1})'(h^{-\alpha}x_{1})Wg\rangle.\end{align*}
We set
$$f=\chi^{i}_{\alpha,\nu}(x,hD)v\quad{}g=\chi^{i}_{\alpha,\nu}(x,hD)u$$
and note that since $\nu^{1/2}_{\alpha,\pm}(x,hD)$ is invertible if 
\begin{equation}\widetilde{W}_{\alpha,\pm}=W\circ{}\nu^{1/2}_{\alpha,\pm}(x,hD)\label{Wtilde:def}\end{equation}
then
\begin{align*}
\langle{}\chi^{i}&_{\alpha,\nu}(x,hD)v,\chi^{i}_{\alpha,\nu}(x,hD)u\rangle_{x'}\\
&=h^{-1}\Big(\langle{}\widetilde{W}_{\alpha}\nu^{-1/2}_{\alpha,\pm}(x,hD)E[\chi^{i}_{\alpha,\nu}(x,hD)v],\theta(x_{1})\chi^{1}(h^{-\alpha}x_{1})\widetilde{W}_{\alpha}\nu^{-1/2}_{\alpha,\pm}(x,hD)\chi^{i}_{\alpha,\nu}(x,hD)u\rangle\\
&+\langle{}\widetilde{W}_{\alpha}\nu^{-1/2}_{\alpha,\pm}(x,hD)\chi^{i}_{\alpha,\nu}(x,hD)v,\theta(x_{1})\chi^{1}(h^{-\alpha}x_{1})\widetilde{W}_{\alpha}\nu^{-1/2}_{\alpha,\pm}(x,hD)E[\chi^{i}_{\alpha,\nu}(x,hD)u]\rangle\Big)\\
&+h^{-\alpha}\langle{}\widetilde{W}_{\alpha}\nu^{-1/2}_{\alpha,\pm}(x,hD)f,\theta(x_{1})(\chi^{1})'(h^{-\alpha}x_{1})\widetilde{W}_{\alpha}\nu^{-1/2}_{\alpha,\pm}g\rangle\end{align*}
Therefore if we can show that for $\alpha\leq{}1/3$
\begin{equation}\norm{\widetilde{W}_{\alpha,\pm}}_{L^{2}(M)\to{}L^{2}(M)}\lesssim{}1\label{Wtildeest}\end{equation}
we obtain \eqref{hypersurfaceinner} as required.

We have
$$W\circ{}\nu^{1/2}_{\alpha,\pm}(x,hD)u=\frac{1}{(2\pi{}h)^{n}}\iint{}e^{\frac{i}{h}(\langle{}x',\xi'\rangle+\phi(x_{1},y,\xi))}\left[\nu^{1/2}_{\alpha,\pm}(y,\nabla_{y}\phi)b(x_{1},y,\xi)+h^{1-2\alpha}r(x,\xi)\right]u(y)d\xi{}dy.$$
Note that since we are in fact estimating $\chi^{1}(h^{-\alpha}x_{1})\widetilde{W}$ we may assume that $|x_{1}|\leq{}\epsilon{}h$. Since we will operate only on functions localised on a $h^{\alpha}$ scale away from $\nu(x,\xi)=0$ we may assume the symbol is also localised on such a (possibly a little larger) region. Further since $\nu^{1/2}_{\alpha,\pm}(x,\xi)>h^{\alpha/2}$ and since $|\alpha|\leq{}1/3$ the error term
$|r(x,\xi)|\leq{}h^{1/3}$ so we can write
$$W\circ{}\nu^{1/2}_{\alpha,+}(x,hD)u=\frac{1}{(2\pi{}h)^{n}}\iint{}e^{\frac{i}{h}(\langle{}x',\xi'\rangle+\phi(x_{1},y,\xi))}\nu^{1/2}_{\alpha,\pm}(y,\nabla_{y}\phi)\tilde{b}(x_{1},y,\xi)u(y)d\xi{}dy$$
where $\tilde{b}(x_{1},y,\xi)\in{}S^{\alpha}$.

We will calculate $U_{\alpha,\pm}=\widetilde{W}_{\alpha,\pm}(\widetilde{W}_{\alpha,\pm})^\star$ and show that it has the form
\begin{equation}U_{\alpha,\pm}u=\frac{1}{(2\pi{}h)^{n}}\iint{}e^{\frac{i}{h}\langle{}x-y,\xi\rangle}b(x,y,\xi)u(y)dyd\xi+O(h^{\infty})\label{WWstar}\end{equation}
with 
$$|\partial_{x_{i}}^{N}b|\cdot|\partial^{N}_{\xi_{i}}b|\leq{}C_{N}h^{-N}.$$ Therefore  from standard results about the $L^{2}\to{}L^{2}$ mapping properties of pseudodifferential operators (see for example \cite{Zworski12})
we obtain
$$\norm{\widetilde{W}_{\alpha,\pm}}_{L^{2}(M)\to{}L^{2}(M)}^{2}=\norm{U_{\alpha,\pm}}_{L^{2}(M)\to{}L^{2}(M)}\lesssim{}1.$$
During our calculations it will often be useful to recall that $\phi$ can be written as
$$\phi(x_{1},y,\xi)=\tlide{\phi}(x_{1},y,\xi)-\langle{}y,\xi\rangle.$$
Now
$$U_{\alpha,\pm} u=\frac{1}{(2\pi{}h)^{2n}}\iint{}e^{\frac{i}{h}(\langle{}x',\xi'\rangle+\phi(x_{1},z,\xi)-\langle{}y',\eta'\rangle-\phi(y_{1},z,\eta))}D(x,y,z,\xi,\eta)u(y)dydzd\xi d\eta$$
where
$$D(x,y,z,\xi,\eta)=\nu^{1/2}_{\alpha,\pm}(z,\nabla_{z}\phi(x_{1},z,\xi))\nu^{1/2}_{\alpha,\pm}(z,\nabla_{z}\phi(y_{1},z,\eta))b(x,y,z,\xi,\eta).$$
We first calculate the $(z,\eta)$ integral using stationary phase. The critical point equations are
\begin{equation}\begin{cases}
\nabla_{z}\phi(x_{1},z,\xi)=\nabla_{z}\phi(y_{1},z,\eta)\\
-y'+\nabla_{\eta'}\phi(y_{1},z,\eta)=0\\
\partial_{\eta_{1}}\phi(y_{1},z,\eta)=0.\end{cases}\label{critpointeq}\end{equation}
Since 
$$\phi(x_{1},y,\xi)=-\langle{}y,\xi\rangle+O(|x_{1}|)$$ this is a non-degenerate critical point. We first make some observations about the critical point
\begin{enumerate}
\item When $x=y$, we require $\xi=\eta$ to satisfy the critical point equations \eqref{critpointeq}.
\item We have $\nabla_{\eta}\phi(y_{1},z,\eta)=z+O(|y_{1}|)$, $\nabla_{\xi}\phi(x_{1},z,\xi)=z+O(|x_{1}|)$ and we may assume $|x_{1}|,|y_{1}|<\epsilon{}h^{\alpha}$. Therefore, if $|x-y|\gg\epsilon{}h^{\alpha}$, we may integrate by parts in either $\xi$ or $\eta$ to get a $h^{\infty}$ error. So we may assume $|x-y|<\epsilon{}h^{\alpha}$.
\item If the critical points are given by $(z(x,y,\xi),\eta(x,y,\xi))$ we have $z(x,y,\xi)=y+O(|y_{1}|+|x_{1}|)$ and $\eta(x,y,\xi)=\xi+O(|y_{1}|+|x_{1}|)$.
\end{enumerate}
We write
$$U_{\alpha,\pm} u=\frac{1}{(2\pi{}h)^{n}}\iint{}e^{\frac{i}{h}\psi(x,,y,\xi)}c(x,y,\xi)u(y)dyd\xi$$
where
\begin{multline}\psi(x,y,\xi)=\langle{}x',\xi'\rangle+\phi(x_{1},z(x,y,\xi),\xi)\\
-\langle{}y',\eta'(x,y,\xi)\rangle-\phi(y_{1},z(x,y,\xi),\eta(x,y,\xi))\label{psialphadef}\end{multline}
and
\begin{equation}c(x,y,\xi)=\nu^{1/2}_{\alpha,\pm}(z(x,y,\xi),\nabla_{z}\phi(x_{1},z(x,y,\xi),\xi))\nu^{1/2}_{\alpha,\pm}(z(x,y,\xi),\nabla_{z}\phi(y_{1},z(x,y,\xi),\eta(x,y,\xi))b(x,y,\xi).\label{calphadef}\end{equation}
Now since $\eta(x,y,\xi)=\xi+O(|x_{1}|+|y_1|)$ and $z=(0,y')+O(|x_{1}|+|y_{1}|)$ we may write
$$c(x,y,\xi)=\nu(0,y',\xi)\tilde{b}(x,y,\xi).$$
Since we know that $\eta(x,x,\xi)=\xi$ we can see from \eqref{psialphadef} that $\psi(x,x,\xi)=0$. So we may write
$$\psi(x,y,\xi)=(x-y)\cdot{}G(x,y,\xi).$$
We make a change first in the dashed variables,
$$\bar{\xi}'=G'(x,y,\xi).$$
To calculate the Jacobian note that since $|x-y|\leq{}\epsilon{}h^{\alpha}$
$$\frac{\partial^{2}\psi}{\partial{}x_{i}\partial_{\xi'_{j}}}=\frac{\partial{}G'_{i}}{\partial\xi_{j}}+O(\epsilon{}h^{\alpha}).$$
So we calculate the mixed derivatives $\partial^{2}_{x,\xi}\psi$. To do this we write $\psi(x,y,\xi)$ as
\begin{align*}\psi(x,y,\xi)&=\langle{}x'-y',\xi'\rangle+\tilde{\phi}(x_{1},z(x,y,\xi),\xi)-\tlide{\phi}(y_{1},z(x,y,\xi),\eta(x,y,\xi))\\
&+\langle{}z(x,y,\xi),\eta(x,y,\xi)-\xi\rangle-\langle{}y',\eta'(x,y,\xi)-\xi'\rangle.\end{align*}
So
$$\psi(x,y,\xi)=\langle{}x'-y',\xi'\rangle+x_{1}p(0,y',\xi)-y_{1}p(0,y',\xi)+O((|x_{1}+|y_{1}|)^{2}).$$
Then
$$\frac{\partial\psi}{\partial_{x_{i}}\partial_{\xi'_{j}}}=\delta_{ij}+O(h^{\alpha})$$
and the Jacobian matrix is given by $\Id+O(h^{\alpha})$. Therefore
$$U_{\alpha,\pm}=\frac{1}{(2\pi{}h)^{n}}\iint{}e^{\frac{i}{h}((x_1-y_1){}G_{1}(x,y,\xi_{1},\xi'(x,y,\xi_{1},\bar{\xi}'))+\langle{}x'-y',\bar{\xi}'\rangle}c(x,y,\xi_{1},\bar{\xi}')u(y)dyd\xi$$
with
$$c(x,y,\xi_{1},\bar{\xi}')=\nu(0,y',\xi_{1},\xi'(x,y,\xi_{1},\bar{\xi}'))b(x,y,\xi_{1},\xi'(x,y,\xi_{1},\bar{\xi}')).$$
Note that since the Jacobian of the transformation is bounded below, $c(x,y,\xi_{1},\bar{\xi}')$ inherits the regularity properties of $c(x,y,\xi)$, that is it is in $S^{\alpha}$. Finally we perform a change of variables in the $\xi_{1}$ coordinates.
$$h^{\alpha}\bar{\xi}_{1}=G_{1}(x,y,\xi_{1},\xi'(x,y,\xi_{1},\bar{\xi}')).$$
As in the case of the dashed coordinates we have that $|x-y|\leq{}\epsilon{}h^{\alpha}$ implies
$$\frac{\partial^{2}\psi}{\partial{}x_{i}\partial_{\xi'_{j}}}=\frac{\partial{}G'_{i}}{\partial\xi_{j}}+O(\epsilon{}h^{\alpha}).$$
So we need to calculate $\frac{\partial^{2}\psi}{\partial{}x_{1}\partial\xi_{1}}$.We rewrite $\psi$ as
$$\psi(x,y,\xi)=x_{1}p(0,y',\xi_{1},\xi'(x,y,\xi_{1},\bar{\xi}'))-y_{1}p(0,y',\xi_{1},\xi'(x,y,\xi_{1},\bar{\xi}'))+O(|x'-y'|)+O((|x_{1}|+|y_{1}|)^{2})$$
and using the fact that $|x-y|\leq{}\epsilon{}h^{\alpha}$ obtain 
$$\frac{\partial\psi}{\partial{}x_{1}\partial\xi_{1}}=\nu(0,y',\xi_{1},\xi'(x,y,\xi_{1},\bar{\xi}'))+O(\epsilon{}h^{\alpha}).$$
So since $\nu$ is localised such that $|\nu(0,y',\xi_{1},\xi'(x,y,\xi_{1},\bar{\xi}'))|>h^{\alpha}$ we have the bound
$$\left|\frac{\partial\bar{\xi}_{1}}{\partial\xi_{1}}\right|>h^{-\alpha}|\nu(0,y',\xi_{1},\xi'(x,y,\xi_{1},\bar{\xi}'))|.$$
Therefore we perform the change of variables and cancel the factor of $\nu$ with that in the symbol to obtain 
$$U_{\alpha,\pm}=\frac{h^{\alpha}}{(2\pi{}h)^{n}}\iint{}e^{\frac{i}{h}(h^{\alpha}(x_{1}-y_{1})\xi_{1}+\langle{}x'-y',\bar{\xi'})\rangle}c(x,y,\bar{\xi})u(y)dyd\xi$$
with
$$c(x,y,\bar{\xi})=c(x,y,\xi_{1}(x,y,\bar{\xi}),\bar{\xi}').$$
Again since $|\frac{\partial\bar{\xi}_{1}}{\partial\xi_{1}}|$ is bounded below $c(x,y,\bar{\xi})$ remains in $S^{\alpha}$. Finally we perform a scaling $h^{\alpha}\bar{\xi}_{1}\to\bar{\xi_{1}}$. This scaling makes the regularity in $\bar{\xi}_{1}$ a bit worse, we have the estimates
\begin{equation}\left|D^{\gamma}_{x,y,\bar{\xi}'}c(x,y,\bar{\xi})\right|\leq{}C_{\gamma}h^{-\alpha|\gamma|}\label{otherderiv}\end{equation}
and
\begin{equation}\left|\frac{\partial^{N}}{\partial^{N}\bar{\xi}_{1}}c(x,y,\bar{\xi})\right|\leq{}C_{N}h^{-2\alpha{}N}.\label{xibarderiv}\end{equation}
Since $\alpha\leq{}1/3$ \eqref{otherderiv} and \eqref{xibarderiv} together imply
$$|\partial_{x_{i}}^{N}c|\cdot{}|\partial_{\xi_{i}}^{N}c|\leq{}C_{N}h^{-N}$$
and therefore have obtained \eqref{WWstar}. So
$$\norm{U_{\alpha,\pm}}_{L^{2}(M)\to{}L^{2}(M)}\lesssim{}1$$
as desired.

 \end{proof}

 We may now prove the strong non-tangential result
 
 \begin{thm}\label{thm:nontang}
 Let $u\in{}L^{2}$ be a $O_{L^{2}}(h)$ quasimode of $p(x,hD)$ and $\alpha\leq{}\frac{1}{3}$.  

Then
 \begin{equation}\norm{\nu(x,hD)\chi^{2}_{\alpha,\nu}(x,hD)u}_{L^{2}(H)}\lesssim{}h^{\frac{\alpha}{2}}\norm{u}_{L^{2}(M)}\label{inner-norm}\end{equation}
 and
\begin{equation}\norm{\nu^{1/2}_{\alpha,\pm}(x,hD)\chi^{2}_{\alpha,\nu}(x,hD)u}_{L^{2}(H)}\lesssim{}\norm{u}_{L^{2}(M)}.\label{sqrt-norm}\end{equation}

 \end{thm}
 
  \begin{proof}
We first cut $\nu(x,hD)\chi^{2}_{\alpha,\nu}(x,hD)u$ off to a region of distance $h^{2\alpha}$ from the hypersurface. We then apply Proposition \ref{prop:nontang} to this function.
In Lemma \ref{lem:nearhypersurface} we will show that
\begin{equation}\norm{\nu^{1/2}_{\alpha,\pm}(x,hD)\chi^{2}_{\alpha,\nu}(x,hD)u}_{L^{2}([-h^{2\alpha},h^{2\alpha}]\times{}\R^{n-1})}\lesssim{}h^{\alpha}\norm{u}_{L^{2}(M)}\label{nuhalfest}\end{equation}
and,
\begin{equation} \norm{\chi^{2}_{\alpha,\nu}(x,hD)u}_{L^{2}([-h^{2\alpha},h^{2\alpha}]\times{}\R^{n-1})}\lesssim{}h^{\frac{\alpha}{2}}\norm{u}_{L^{2}(M)}.\label{chiest}\end{equation}
In fact these estimates hold for a $h^{2\alpha}$ thickened neighbourhood of any hypersurface $x_1=K$. For the moment we will take \eqref{nuhalfest} and \eqref{chiest} as given. Let
$$v=\chi^{1}(h^{-2\alpha}x_{1})\nu(x,hD)\chi^{2}_{\alpha,\nu}(x,hD)u.$$
Using Proposition \ref{prop:nontang} it would be enough to obtain
$$\norm{\nu_{\alpha,\pm}^{-1/2}(x,hD)\chi^{2}_{\alpha,\nu}(x,hD)v}_{L^{2}(M)}\lesssim{}h^{\alpha}\norm{u}_{L^{2}(M)}$$
and
$$\norm{\nu_{\alpha,\pm}^{-1/2}(x,hD)E[\chi^{2}_{\alpha,\nu}(x,hD)v]}_{L^{2}(M)}\lesssim{}h\norm{u}_{L^{2}(M)}.$$
Now
\begin{align*}
\nu_{\alpha,\pm}^{-1/2}(x,hD)&\chi^{2}_{\alpha,\nu}(x,hD)\chi^{1}(h^{-2\alpha}x_{1})\nu(x,hD)\chi^{2}_{\alpha,\nu}(x,hD)u\\
&=\chi^{2}_{\alpha,\nu}(x,hD)\left(\nu^{-1/2}_{\alpha,\pm}(x,hD)\chi^{1}(h^{-2\alpha}x_{1})\nu^{1/2}_{\alpha,\pm}(x,hD)\right)\nu^{1/2}_{\alpha,\pm}(x,hD)\chi^{2}_{\alpha,\nu}(x,hD)u\\
&+O\left(h^{1-\alpha/2}\norm{u}_{L^{2}(M)}\right).\end{align*}
So by Lemma \ref{lem:comm} 
\begin{align*}
&\norm{\nu_{\alpha,\pm}^{-1/2}(x,hD)\chi^{2}_{\alpha,\nu}(x,hD)\chi^{1}(h^{-2\alpha}x_{1})\nu(x,hD)\chi^{2}_{\alpha,\nu}u}_{L^{2}(M)}\\
&\lesssim\sup_{K\in[-1,1]}\norm{\nu^{1/2}(x,hD)\chi^{2}_{\alpha,\nu}(x,hD)u}_{L^{2}([K-h^{2\alpha},K+h^{2\alpha}]\times\R^{n-1})}+O(h^{1-\alpha/2}\norm{u}_{L^{2}(M)})\\
&\lesssim{}h^{\alpha}\norm{u}_{L^{2}}\end{align*}
as required. Now turning to the those terms involving an $E[\chi^{2}_{\alpha,\nu}(x,hD)v]$ term. 
$$E[\chi^{2}_{\alpha,\nu}(x,hD)v]=\chi^{2}_{\alpha,\nu}(x,hD)E[v]+h^{1-\alpha}r(x,hD)v.$$
Since $\nu(x,\xi)$ is localised to size $h^{\alpha}$ and $v$ is supported only in the $h^{2\alpha}$ thickened hypersurface $x_{1}=0$ Lemma \ref{lem:nearhypersurface} gives us
$$\norm{v}_{L^{2}(M)}\lesssim{}h^{\frac{3\alpha}{2}}\norm{u}_{L^{2}(M)},$$
so we may focus on the term involving $E[v]$. 
\begin{align*}
E[v]&=\chi^{1}(h^{-2\alpha}x_{1})\nu(x,hD)\chi^{2}_{\alpha,\nu}(x,hD)p(x,hD)u\\
&+h^{1-2\alpha}(\chi^{1})'(h^{-2\alpha}x_{1})\nu^{2}(x,hD)\chi^{2}_{\alpha,\nu}(x,hD)u\\
&+h\chi^{1}(h^{-2\alpha}x_{1})\dot{\nu}(x,hD)\left[\chi^{2}_{\alpha,\nu}(x,hD)+h^{-\alpha}(\chi^{2})'_{\alpha,\nu}(x,hD)\nu(x,hD)\right]u\\
&+O\left(h^{2-4\alpha}\norm{\nu(x,hD)\chi^{2}_{\alpha,\nu}(x,hD)u}_{L^{2}(M)}+h^{2-2\alpha}\norm{u}_{L^{2}(M)}\right).\end{align*}
The fourth term is $O(h\norm{u}_{L^{2}(M)})$ so we focus on the others. Since $u$ is an $O_{L^{2}}(h)$ quasimode of $p(x,hD)$ we have
$$\norm{\chi^{1}(h^{-2\alpha}x_{1})\nu(x,hD)\chi^{2}_{\alpha,\nu}(x,hD)p(x,hD)u}_{L^{2}(M)}\lesssim{}h^{1+\alpha}\norm{u}_{L^{2}(M)}$$
therefore
$$\norm{\nu^{-1/2}_{\alpha,\pm}(x,hD)\chi^{1}(h^{-2\alpha}x_{1})\nu(x,hD)\chi^{2}_{\alpha,\nu}(x,hD)p(x,hD)u}_{L^{2}(M)}\lesssim{}h^{1+\alpha/2}\norm{u}_{L^{2}(M)}$$
which is better than we require. To treat the second term note (using Lemmas \ref{lem:comm} and \ref{lem:nearhypersurface})
\begin{align*}h^{1-2\alpha}\nu^{-1/2}_{\alpha,\pm}&(x,hD)(\chi^{1})'(h^{-2\alpha}x_{1})\nu^{2}(x,hD)\chi^{2}_{\alpha,\nu}(x,hD)\\
&h^{1-2\alpha}\left(\nu^{-1/2}_{\alpha,\pm}(x,hD)(\chi^{1})'(h^{-2\alpha}x_{1})\nu^{1/2}_{\alpha,\pm}(x,hD)\right)\nu^{3/2}(x,hD)\chi^{2}_{\alpha,\nu}(x,hD)u+O(h^{2-5\alpha/2}\norm{u}_{L^{2}(M)})\\
&\lesssim{}h^{1-\alpha/2}\sup_{K\in[-1,1]}\norm{\chi^{2}_{\alpha,\nu}(x,hD)u}_{L^{2}([K-h^{2\alpha},K+h^{\alpha}]\times\R^{n-1})}\\
&\lesssim{}h\norm{u}_{L^{2}(M)}.\end{align*}
Finally
\begin{align*}h\nu^{-1/2}_{\alpha,\pm}&(x,hD)\chi^{1}(h^{-2\alpha}x_{1})\dot{\nu}(x,hD)\left[\chi^{2}_{\alpha,\nu}(x,hD)+h^{-\alpha}(\chi^{2})'_{\alpha,\nu}(x,hD)\nu(x,hD)\right]u\\
&=h\nu^{-1/2}_{\alpha,\pm}(x,hD)\chi^{1}(h^{-2\alpha}x_{1})\dot{\nu}(x,hD)\chi^{2}_{\alpha,\nu}(x,hD)u\\
&+h^{1-\alpha}\left(\nu^{-1/2}_{\alpha,\pm}(x,hD)\chi^{1}(h^{-2\alpha}x_{1})\nu^{1/2}_{\alpha,\pm}(x,hD)\right)\nu^{1/2}_{\alpha,\pm}(x,hD)\dot{\nu}(x,hD)(\chi^{2})'_{\alpha,\nu}(x,hD)u\\
&+O(h^{2-5\alpha/2}\norm{u}_{L^{2}(M)})\end{align*}
and so by applying Lemma \ref{lem:comm} and Theorem \ref{commutatorthm} to the first two terms we obtain
$$\norm{h\nu^{-1/2}_{\alpha,\pm}(x,hD)\chi^{1}(h^{-2\alpha}x_{1})\dot{\nu}(x,hD)\left[\chi^{2}_{\alpha,\nu}(x,hD)+h^{-\alpha}(\chi^{2})'_{\alpha,\nu}(x,hD)\nu(x,hD)\right]u}_{L^{2}(M)}\lesssim{}h\norm{u}_{L^{2}(M)}$$
Therefore we have
$$\norm{\nu(x,hD)\chi^{2}_{\alpha,\nu}(x,hD)u}_{L^{2}(H)}\lesssim{}h^{\alpha/2}\norm{u}_{L^{2}(M)}$$
and since on the support of $\chi^{2}$, $|\nu(x,\xi)|\geq h^{\alpha}/2$,
$$\norm{\nu^{1/2}_{\alpha,\pm}(x,hD)\chi^{2}_{\alpha,\nu}(x,hD)u}_{L^{2}(H)}\lesssim{}\norm{u}_{L^{2}(M)}.$$
\end{proof}

\begin{lem}\label{lem:nearhypersurface}
 Suppose $u$ is an $O_{L^{2}}(h)$ quasimode of $p(x,hD)$ then for $\alpha\leq{}1/3$ 
 \begin{equation}
  \norm{\nu^{1/2}_{\alpha,\pm}(x,hD)\chi^{2}_{\alpha,\nu}(x,hD)u}_{L^{2}([-h^{2\alpha},h^{2\alpha}]\times{}\R^{n-1})}\lesssim{}h^{\alpha}\norm{u}_{L^{2}(M)}\label{nearhypernuhalf}\end{equation}
  and
  \begin{equation}
  \norm{\chi^{2}_{\alpha,\nu}(x,hD)u}_{L^{2}([-h^{2\alpha},h^{2\alpha}]\times{}\R^{n-1})}\lesssim{}h^{\frac{\alpha}{2}}\norm{u}_{L^{2}(M)}\quad{}\label{nearhyperu}\end{equation}
   Further       \begin{equation}
  \norm{\nu(x,hD)\chi^{1}_{1/3,\nu}(x,hD)u}_{L^{2}([-h^{2/3},h^{2/3}]\times{}\R^{n-1})}\lesssim{}h^{\frac{1}{2}}\norm{u}_{L^{2}(M)}.\label{nearhypertang}\end{equation}

  \end{lem}
  
  \begin{proof}
  The proof of this is similar to that of Theorem \ref{commutatorthm}, however to obtain the finer cut-off we need to consider more terms in the expansion for the commutator symbol. For this finer analysis the quantisation procedure matters. We will use the Weyl quantisation. Let
  $$q(x,hD)^{w}u=\frac{1}{(2\pi{}h)^{n}}\iint{}e^{\frac{i}{h}\langle{}x-y,\xi\rangle}q\left(\frac{x+y}{2},\xi\right)u(y)dyd\xi.$$
  That is $q(x,hD)^{w}$ is the operator obtained under the Weyl quantisation procedure. Note that if $q\in{}S^{\alpha}$, 
  \begin{multline*}q(x,hD)^{w}u=\frac{1}{(2\pi{}h)^{n}}\iint{}e^{\frac{i}{h}\langle{}x-y,\xi\rangle}q(x,\xi)u(y)dyd\xi\\
  +\frac{h^{-\alpha}}{(2\pi{}h)^{n}}\iint{}e^{\frac{i}{h}\langle{}x-y,\xi\rangle}(x-y)\cdot{}r(x,y,\xi)u(y)dyd\xi\end{multline*}
  and integration by parts in $\xi$ tells us that
  $$q(x,hD)^{w}u=q(x,hD)u+O(h^{1-2\alpha}\norm{u}_{L^{2}(M)}).$$
 Incidentally this relationship holds for any two choices of quantisation procedure.  Consequently it is enough to establish \eqref{nearhypernuhalf}, \eqref{nearhyperu} and \eqref{nearhypertang} for the Weyl quantisation.  Let $\zeta(r)$ be such that $\zeta'(r)=(\chi^{1}(r))^{2}$. We will calculate the commutator
 $$[p(x,hD)^{w},\zeta(h^{-2\alpha}x_{1})]$$
 using the Weyl composition formula
 $$a(x,hD)^{w}\circ{}b(x,hD)^{w}=(a\#b)(x,hD)^{w}$$
 where
 $$a\#b(x,\xi)=e^{\frac{ih}{2}\left(\langle{}D_{\xi},D_{y}\rangle-\langle{}D_{x},D_{\eta}\rangle\right)}a(x,\xi)b(y,\eta)\Big|_{x=y,\xi=\eta.}$$
 The key point is that $\zeta$ is a function of $x_{1}$ alone and so terms involving even ordered derivatives cancel out. Therefore we obtain
 $$[p(x,hD)^{w},\zeta(h^{-2\alpha}x_{1})]f=ih^{1-2\alpha}(\chi^{1}(h^{-2\alpha}x_{1}))^{2}\nu(x,hD)^{w}f+h^{3-6\alpha}r(x,hD)f.$$
 Note that this is one power of $h^{1-2\alpha}$ better than we would obtain for any other quantisation procedure.  
So we have
\begin{multline}\langle{}g,(\chi^{1}(h^{-2\alpha}x_{1}))^{2}\nu(x,hD)^{w}f\rangle=h^{-1+2\alpha}\Big(\langle{}p(x,hD)^{w}g,\zeta(h^{-2\alpha}x_{1})f\rangle-\langle{}g,\zeta(h^{-2\alpha}x_{1})p(x,hD)^{w}f\rangle\Big)\\
+O(h^{2-4\alpha})\norm{f}_{L^{2}(M)}\norm{g}_{L^{2}(M)}.\end{multline}
To obtain \eqref{nearhypernuhalf} set $f=\nu^{-1/2}_{\alpha,\pm}(x,hD)^{w}\chi^{2}_{\alpha,\nu}(x,hD)^{w}u$ and $g=\nu^{1/2}_{\alpha,\pm}(x,hD)^{w}\chi^{2}_{\alpha,\nu}(x,hD)^{w}u$.  Now
\begin{multline*}
p(x,hD)^{w}g=\nu_{\alpha,\pm}^{1/2}(x,hD)^{w}\chi^{2}_{\alpha,\nu}(x,hD)^{w}p(x,hD)^{w}u\\
+h\nu^{-1/2}_{\alpha,\pm}(x,hD)^{w}\dot{\nu}(x,hD)^{w}r(x,hD)^{w}u+O(h^{2-2\alpha}\norm{u}_{L^{2}(M)}).\end{multline*}
where the symbol $r(x,\xi)$ is supported in the region $h^{\alpha}/2\leq |\nu(x,\xi)|\leq 5h^{\alpha}/2$. So
\begin{align}\begin{split}h^{-1+2\alpha}&\langle{}p(x,hD)^{w}g,\zeta(h^{-2\alpha}x_{1})f\rangle\\
&=h^{-1+2\alpha}\langle{}\nu_{\alpha,\pm}^{1/2}(x,hD)^{w}\chi^{2}_{\alpha,\nu}(x,hD)^{w}p(x,hD)^{w}u,\zeta(h^{-2\alpha}x_{1})\nu_{\alpha,\pm}^{-1/2}(x,hD)^{w}\chi^{2}_{\alpha,\nu}(x,hD)^{w}u\rangle\\
&+h\langle\nu^{-1/2}_{\alpha,\pm}(x,hD)^{w}\dot{\nu}(x,hD)^{w}r(x,hD)u,\nu^{-1/2}_{\alpha,\pm}(x,hD)^{w}\chi^{2}_{\alpha,\nu}(x,hD)u\rangle+O(h^{1-\alpha/2}\norm{u}_{L^{2}(M)}^{2}).\label{expandedprod}\end{split}\end{align}
The error term $O(h^{1-\alpha/2}\norm{u}_{L^{2}(M)}^{2})$ is better than required so we focus on the first two terms. We have
\begin{multline*}h^{-1+2\alpha}\langle{}\nu_{\alpha,\pm}^{1/2}(x,hD)^{w}\chi^{2}_{\alpha,\nu}(x,hD)p(x,hD)^{w}u,\zeta(h^{-2\alpha}x_{1})\nu_{\alpha,\pm}^{-1/2}(x,hD)^{w}\chi^{2}_{\alpha,\nu}(x,hD)^{w}u\rangle\\
=h^{-1+2\alpha}\langle{}\chi^{2}_{\alpha,\nu}(x,hD)^{w}p(x,hD)^{w}u,\nu^{1/2}_{\alpha,\pm}(x,hD)^{w}\zeta(h^{-2\alpha}x_{1})\nu^{-1/2}_{\alpha,\pm}(x,hD)^{w}\chi^{2}_{\alpha,\nu}(x,hD)^{w}u\rangle.\end{multline*}
Since
$$\nu^{\pm 1/2}_{\alpha,\pm}(x,hD)^{w}=\nu^{\pm 1/2}_{\alpha,\pm}(x,hD)+O_{L^{2}(M)\to{}L^{2}(M)}(h^{1-2\alpha})$$
and Lemma \ref{lem:comm} tells us that $\nu^{1/2}_{\alpha,\pm}(x,hD)\zeta(h^{-2\alpha}x_{1})\nu^{-1/2}_{\alpha,\pm}(x,hD)$ is a bounded operator from $L^{2}(M)\to{}L^{2}(M)$. We conclude that
$$h^{-1+2\alpha}\left|\langle{}\nu_{\alpha,\pm}^{1/2}(x,hD)^{w}p(x,hD)^{w}u,\zeta(h^{-2\alpha}x_{1})\nu_{\alpha,\pm}^{-1/2}(x,hD)^{w}u\rangle\right|\lesssim{}h^{2\alpha}\norm{u}_{L^{2}(M)}^{2}.$$
Now consider the second term in \eqref{expandedprod} we have
\begin{align*}
h^{2\alpha}&\langle{}\nu^{-1/2}(x,hD)^{w}\dot{\nu}(x,hD)^{w}r(x,hD)^{w}u,\zeta(h^{-2\alpha}x_{1})\nu^{-1/2}_{\alpha,\pm}(x,hD)^{w}\chi^{2}_{\alpha,\nu}(x,hD)u\rangle\\
&=h^{2\alpha}\langle{}\nu^{-1/2}(x,hD)v_{\alpha},\dot{\nu}^{1/2}_{\alpha,\pm}(x,hD)\zeta(h^{-2\alpha}x_{1})\dot{\nu}^{-1/2}_{\alpha,\pm}(x,hD)\nu^{-1/2}_{\alpha,\pm}(x,hD)\tilde{v}_{\alpha}\rangle\\
&+O\left(h\norm{u}_{L^{2}(M)}\right)\end{align*}
where both $v_{\alpha}$ and $\tilde{v}_{\alpha}$ have the form $\dot{\nu}^{1/2}_{\alpha,\pm}(x,hD)\rho(x,hD)$ with symbol $\rho(x,\xi)$ supported where $h^{\alpha}/2\leq|\nu(x,\xi)|\leq 5h^{\alpha}/2$.  So by Corollary \ref{cor:sqrt}
$$\norm{v_{\alpha}}_{L^{2}(M)}\lesssim{}h^{\alpha/2}\norm{u}_{L^{2}(M)}\quad{}\norm{\tilde{v}_{\alpha}}_{L^{2}(M)}\lesssim{}h^{\alpha/2}\norm{u}_{L^{2}(M)}$$
and
$$\norm{\nu^{-1/2}(x,hD)^{w} v_{\alpha}}_{L^{2}(M)}\lesssim{}\norm{u}_{L^{2}(M)}\quad{}\norm{\nu^{-1/2}(x,hD)^{w} \tilde{v}_{\alpha}}_{L^{2}(M)}\lesssim{}\norm{u}_{L^{2}(M)}.$$
Therefore  by Lemma \ref{lem:comm}
\begin{align*}h^{2\alpha}&\left|\langle{}\nu^{-1/2}(x,hD)^{w}\dot{\nu}(x,hD)^{w}r(x,hD)^{w}u,\zeta(h^{-2\alpha}x_{1})\nu^{-1/2}_{\alpha,\pm}(x,hD)\chi^{2}_{\alpha,\nu}(x,hD)u\rangle\right|\\
&\lesssim{}h^{2\alpha}\norm{\nu^{-1/2}_{\alpha,\pm}(x,hD)v_{\alpha}}_{L^{2}(M)}\norm{\nu^{-1/2}_{\alpha,\pm}(x,hD)\tilde{v}_{\alpha}}_{L^{2}(M)}+O(h^{1-\alpha}\norm{u}_{L^{2}}^{2})\\
&\lesssim{}h^{2\alpha}\norm{u}_{L^{2}}^{2}.\end{align*}
A similar argument gives
$$\left|\langle{}g,\zeta(h^{-2\alpha}x_{1})p(x,hD)^{w}f\rangle\right|\lesssim{}h^{2\alpha}\norm{u}_{L^{2}}^{2}$$ and therefore we arrive at
$$\norm{\nu_{\alpha,\pm}^{1/2}(x,hD)\chi^{2}_{\alpha,\nu}(x,hD)u}_{{L^{2}([-h^{2\alpha},h^{2\alpha}]\times{}\R^{n-1})}}\lesssim{}h^{\alpha}\norm{u}_{L^{2}(M)}.$$
To obtain \eqref{nearhyperu} and \eqref{nearhypertang} we set 
$$f=(\nu(x,hD)^{w})^{-1}\chi^{2}_{\alpha,\nu}(x,hD)^{w}u\quad{}g=\chi^{2}_{\alpha,\nu}(x,hD)^{w}u$$
and
$$f=\chi^{1}_{\alpha,\nu}(x,hD)^{w}u\quad{}g=\nu(x,hD)^{w}\chi^{1}_{\alpha,\nu}(x,hD)^{w}u$$
respectively and apply a similar argument.
\end{proof}

 Finally we complete the proof of Theorem \ref{semiclassicalthm} by proving that the tangential contribution is bounded. 
 \begin{thm}\label{thm:tang}
Let $u$ be an $O_{L^{2}}(h)$ quasimode of $p(x,hD)$, 
\begin{equation}\norm{\nu(x,hD)\chi^{1}_{1/3,\nu}(x,hD))}_{L^{2}(H)}\lesssim{}\norm{u}_{L^{2}(M)}.\label{general-tang}\end{equation}
\end{thm}

\begin{proof}
From Lemma \ref{lem:nearhypersurface} we know that $$\norm{\chi^{1}(h^{-2/3})\nu(x,hD)\chi^{1}_{\alpha,\nu}(x,hD)v}_{L^{2}(M)}\lesssim{}h^{1/2}\norm{u}_{L^{2}(M)}.$$
Let $v=\nu(x,hD)\chi^{1}_{\alpha,\nu}(x,hD)u$, $u$ is semiclassically localised therefore there exists a $\chi(x,\xi)$ compactly supported so that
$$v=\nu(x,hD)\chi^{1}_{\alpha,\nu}(x,hD)\chi(x,hD)u+O(h^{\infty}).$$
Since the support of $\chi$ is compact there is a point $(x_{0},\xi_{0})$ and a $R>0$ such that $\chi\equiv{}0$ outside $B_{R}(x_{0},\xi_{0})$. Let $\tilde{\chi}$ be a smooth cut off function defined so that
$$\chi(x,\xi)=\begin{cases}
1&|(x,\xi)-(x_{0},\xi_{0})|\leq{}2R\\
0&|(x,\xi)-(x_{0},\xi_{0})|\geq{}4R.\end{cases}$$
Then $\tilde{\chi}(x,hD)v=v+O(h^{\infty})$ since the derivatives of $\tilde{\chi}$ in the composition formula \eqref{semiexp} are zero on the support of $\chi(x,\xi)$. Finally let
$$\rho(x,\xi)=\begin{cases}
1&|(x,\xi)-(x_{0},\xi_{0})|\leq{}6R \text{ and }|x_{1}|\leq{}1\\
0&|(x,\xi)-(x_{0},\xi_{0})|\geq{}8R\text{ or }|x_{1}|\geq{}2.\end{cases}$$
Then $\rho(x,hD)\chi^{1}(h^{-2/3}x_{_1})v=\chi^{1}(h^{-2/3}x_{1})v+O(h^{\infty})$ since similarly derivatives of $\rho$ are zero on the support of $\chi^{1}(h^{-2/3}x_{1})\tilde{\chi}(x,\xi)$. Therefore $\chi^{1}(h^{-2/3}x_{1})v$ is semiclassically localised and  from the standard semiclassical Sobolev estimates
$$\norm{v}_{L^{2}(H)}\lesssim{}h^{-1/2}\norm{\chi^{1}(h^{-2/3}x_{1})v}_{L^{2}(M)}\lesssim{}\norm{u}_{L^{2}(M)}.$$

\end{proof}
  
  Where the symbol $p(x,\xi)$ arises from a Laplacian or similar operator we can do better. 
 
\begin{defin}\label{laplacelike}
A semiclassical pseudodifferential operator is Laplace-like with respect to a hypersurface $H$ if in local coordinates
$$|\partial_{\xi_{1}}\nu(x,\xi)|>c>0.$$
\end{defin}

\begin{thm}\label{thm:LL}
Let $u$ be an $O_{L^{2}}(h)$ quasimode of $p(x,hD)$, which is Laplace-like. Then

\begin{equation}\norm{\nu(x,hD)\chi^{1}_{1/3,\nu}(x,hD)u}_{L^{2}(H)}\lesssim{}h^{\frac{1}{6}}\norm{u}_{L^{2}(M)}.\label{LL-tang-norm}\end{equation}
\end{thm}

\begin{proof}
The proof is similar to that of Theorem \ref{thm:tang}. However if $\nu_{\xi}(x,\xi)$ is bounded away from zero, the restriction of $\nu(x,\xi)$ to an order $h^{1/3}$ region restricts $\xi_{1}$ to a $h^{1/3}$ region. Suppose that $\tilde{\chi}^{1}(r)$ has a slightly larger support that $\chi^{1}(r)$ (specifically $\tilde{\chi}^{1}\equiv{}1$ on the support of $\chi^{1}$), then
$$\tilde{\chi}^{1}_{1/3,\nu}(x,hD)v=v+O(h^{\infty}).$$
That is
$$v(0,x')=\frac{1}{(2\pi{}h)^{n}}\int{}e^{\frac{i}{h}(-y_{1}\xi_{1}+\langle{}x'-y',\xi'\rangle)}\tilde{\chi}^{1}_{1/3,\nu}v(y)dyd\xi.$$
Note that when $|y_{1}|>Kh^{2/3}$ for some suitably large $K$ we can integrate by parts in $\xi_{1}$ and obtain decay of $h^{2N/3}|y|^{-N}$ for any $N$. Therefore we can say that
$$v(0,x')=\frac{1}{(2\pi{}h)^{n}}\int{}e^{\frac{i}{h}(-y_{1}\xi_{1}+\langle{}x'-y',\xi'\rangle)}\rho^{1}_{1/3,\nu}\zeta(h^{-2/3}y_{1})v(y)dyd\xi.$$
where $\rho$ has the same support properties as $\tilde{\chi}^{1}$ and $\zeta(r)$ decays like $(1+|r|)^{-N}$. Now since $|\partial_{\xi_{1}}\nu(x,\xi)|>c>0$ we may the change of variables $\bar{\xi}_{1}=\nu(x,\xi)$ to obtain
$$v(0,x')=\frac{1}{(2\pi{}h)^{n}}\int{}e^{\frac{i}{h}(-y_{1}\xi_{1}(x,\bar{\xi}_{1},\xi')+\langle{}x'-y',\xi'\rangle)}\rho^{1}(h^{-1/3}\bar{\xi}_{1})\zeta(h^{-2/3}y_{1})v(y)dyd\xi.$$
Now applying standard $L^{2}$ estimates for the dashed variables we obtain
$$\norm{v}_{L^{2}(H)}\lesssim{}h^{-1/3}\norm{\zeta(h^{-2/3}x_{1})v}_{L^{2}(M)}.$$
Since Lemma \ref{lem:nearhypersurface} tells us that
$$\norm{\zeta(h^{-2/3}x_{1})v}_{L^{2}}\lesssim{}h^{1/2}\norm{u}_{L^{2}}$$
we obtain
$$\norm{v}_{L^{2}(H)}\lesssim{}h^{1/6}\norm{u}_{L^{2}(M)}.$$
\end{proof}

\section{Curved Hypersurfaces}\label{sec:curved}

Theorem \ref{commutatorthm} is independently interesting and actually allows us to reproduce results on restriction estimates of curved hypersurfaces in an elementary fashion. These results (due to Tataru \cite{tataru98} and Hu \cite{Hu} for Laplacians and Hassell-Tacy \cite{HTacy} for semiclassical operators) state that under conditions.
\begin{enumerate}
\item For any point $(x_{0},\xi_{0})$ such that $p(x_{0},\xi_{0})=0$; $\nabla_{\xi}p(x_{0},\xi_{0})\neq{}0.$
\item The hypersurface $\{\xi\mid{}p(x_{0},\xi)=0\}$ has positive definite fundamental form
\item For a boundary defining function $r$ we have $\dot{r}(x_{0},\xi_{0})=0\Rightarrow{}\ddot{r}(x_{0},\xi_{0})\neq{}0$,
\end{enumerate}
quasimodes $u$ of order $h$ obey
$$\norm{u}_{L^{2}(H)}\lesssim{}h^{-1/6}\norm{u}_{L^{2}(M)}$$
which is an improvement over the standard $h^{-1/4}$ bound (that holds when we assume only (1) and (2)).

Since the function $x_{1}$ is a boundary defining function for $H$ and 
$$\dot{x}_{1}(x,\xi)=\nu(x,\xi)\quad\ddot{x_{1}}(x,\xi)=\{p,\nu\}(x,\xi)$$ 
the third condition is equivalent to stating that $|\dot{\nu}(x,\xi)|>c$ whenever we are localised around a point $(x_{0},\xi_{0})$ such that $\nu(x_{0},\xi_{0})=0$. If we are localised about a point where $\nu(x_{0},\xi_{0})\neq{}0$ we can treat the restriction of $u$ to $H$ as in Theorem \ref{thm:nontang} and obtain
$$\norm{u}_{L^{2}(H)}\lesssim{}\norm{u}_{L^{2}(M)}$$
so any concentration must come for regions localised around $\nu(x_{0},\xi_{0})=0$. The second condition implies $\partial_{\xi_{1}}\nu(x,\xi)=\partial^{2}_{\xi_{1}\xi_{1}}p(x,\xi)$ is bounded away from zero (as it is for the Laplacian case where $\partial_{\xi_{1}\xi_{1}}^{2}\nu(x,\xi)=2$). 

\begin{prop}\label{curvedest}
Suppose $|\partial_{\xi_{1}}\nu(x,\xi)|$ and $|\dot{\nu}(x,\xi)|$ are bounded away from zero. Then if $u$ is an $O_{L^{2}}(h)$ quasimode of $p(x,hD)$
$$\norm{u}_{L^{2}(H)}\lesssim{}h^{-1/6}\norm{u}_{L^{2}(M)}.$$
\end{prop}

\begin{proof}
Let $\alpha\leq1/3$ by Theorem \ref{thm:nontang} we know that
$$\norm{\nu(x,hD)\chi^{2}_{\alpha,\nu}(x,hD)u}_{L^{2}(H)}\lesssim h^{\alpha/2}\norm{u}_{L^{2}(M)}$$
therefore, as on the support of $\chi^{2}_{\alpha,\nu}(x,\xi)$, $|\nu(x,\xi)|>h^{\alpha}$ 
$$\norm{\chi^{2}_{\alpha,\nu}(x,hD)u}_{L^{2}(H)}\lesssim{}h^{-\alpha/2}\norm{u}_{L^{2}(M)}.$$ 
Similarly 
$$\norm{\chi^{2}_{\alpha,-\nu}(x,hD)u}_{L^{2}(H)}\lesssim{}h^{-\alpha/2}\norm{u}_{L^{2}(M)}.$$ 
So by dyadically summing 
$$\norm{\chi^{3}_{\alpha,\nu}(x,hD)u}_{L^{2}(H)}+\norm{\chi^{3}_{\alpha,-\nu}(x,hD)u}_{L^{2}(H)}\lesssim{}h^{-1/6}\norm{u}_{L^{2}(M)}.$$
Finally we must consider the tangential term. By Theorem \ref{commutatorthm}
$$\norm{\dot{\nu}(x,hD)\chi^{1}_{1/3.\nu}(x,hD)u}_{L^{2}(M)}\lesssim{}h^{1/6}\norm{u}_{L^{2}(M)}$$
and since assumption (3) ensures $|\dot{\nu}(x,\xi)|>c>0$ we may invert $\dot{\nu}(x,hD)$ to obtain
$$\norm{\chi^{1}_{1/3,\nu}(x,hD)u}_{L^{2}(M)}\lesssim{}h^{1/6}\norm{u}_{L^{2}(M)}.$$
Now clearly for the Laplacian itself we may apply the results of Theorem \ref{thm:LL} to obtain
$$\norm{\chi^{1}_{1/3,\nu}(x,hD)u}_{L^{2}(H)}\lesssim{}h^{-1/3}\norm{\chi^{1}_{1/3,\nu}(x,hD)}_{L^{2}(M)}\lesssim{}h^{-1/6}\norm{u}_{L^{2}(M)}.$$
\end{proof}

\section{Saturation and examples}\label{sec:sharp}

We will study four examples to illustrate sharpness. In $T^{\star}\R^{n}$ we write a point $(x,\xi)$ as $(x_{1},x_{2},\bar{x},\xi_{1},\xi_{2},\bar{\xi})$, and denote $x'=(x_{1},\bar{x})$, $\xi'=(\xi_{1},\bar{\xi})$. For these examples the hypersurface is always $\{x\mid x_{1}=0\}$. 

\begin{enumerate}
\item $p(x,\xi)=\xi_{2}-|\xi'|^{2}$
\item $p(x,\xi)=\xi_{2}-|\xi'|^{2}-x_{1}$
\item $p(x,\xi)=\xi_{1}\xi_{2}$
\item $p(x,\xi)=\xi_{1}\xi_{2}-x_{2}$
\end{enumerate}
For all the examples we will use the semiclassical Fourier transform
$$\mathcal{F}_{h}u=\frac{1}{(2\pi{}h)^{n/2}}\int{}e^{\frac{i}{h}x\cdot\xi}u(y)dyd\xi$$
to construct quasimodes. With this scaling $\mathcal{F}_{h}$ has the nice property that
$$\norm{u}_{L^{2}}=\norm{\mathcal{F}_{h}u}_{L^{2}}$$
so we may solve on the Fourier side and then invert to produce $u$. 

\begin{example}\label{ex1}
Let $p(x,\xi)=\xi_{2}-|\xi'|^{2}$, then $\nu(x,\xi)=-2\xi_{1}$. This is a model for the flat Laplacian localised in a region where $|\xi_{2}|\sim{}1$. Taking the Fourier transform we find that a quasimode of $p(x,hD)$ has
$$(\xi_{2}-|\xi'|^{2})\mathcal{F}_{h}u=O_{L^{2}}(h)$$
we write $\xi=(\xi_{1},\xi_{2},\bar{\xi})$. Let
$$\chi_{\alpha}(\xi)=\begin{cases}
1&|\xi_{2}-\xi_{1}^{2}|\leq{}h,h^{\alpha}\leq{}\xi_{1}\leq{}2h^{\alpha},|\bar{\xi}|\leq{}h^{1/2}\\
0&\mbox{otherwise.}\end{cases}.$$
Now set
$$f_{\alpha}(\xi)=h^{-1/2-\alpha/2-(n-2)/4}\chi_{\alpha}(\xi)$$
which is $L^{2}$ normalised and let $u_{\alpha}=\mathcal{F}_{h}^{-1}f_{\alpha}$. Note that $u_{\alpha}$ is localised where $\nu(x,\xi)\sim{}h^{\alpha}$. Finally we have
$$u_{\alpha}|_{H}=\frac{h^{-1/2-\alpha/2-(n-2)/4}}{(2\pi{}h)^{n/2}}\int{}e^{\frac{i}{h}x'\cdot{}\xi'}\chi_{\alpha}(\xi)d\xi.$$
Note that for $|\bar{x}|\leq{}h^{1/2}$ the term $e^{\frac{i}{h}\langle{}\bar{x},\bar{\xi}\rangle}$ does not oscillate very much.  Similarly for $|x_{2}|\leq{}h^{1-2\alpha}$ the $e^{\frac{i}{h}x_{2}\cdot\xi_{2}}$ term does not oscillate. So for
$|x_{2}|\leq{}h^{2\alpha}$, $|\bar{x}|\leq{}h^{1/2}$ we have
$$|u_{\alpha}|>h^{-1/2-\alpha/2-(n-2)/4-n/2}\cdot{}h^{1+\alpha+(n-2)/2}>h^{1/2+\alpha/2+(n-2)/4-n/2}$$
and
$$\norm{u_{\alpha}}_{L^{2}_{x'}}>h^{1/2+\alpha/2+(n-2)/4-n/2}h^{1/2-\alpha}h^{(n-2)/4}=h^{-\alpha/2}.$$
Therefore
$$\norm{\nu(x,hD)u_{\alpha}}_{L^{2}(H)}>ch^{\alpha/2}.$$
\end{example}

\begin{example}\label{ex2}
Let $p(x,\xi)=\xi_{2}-|\xi'|^{2}-x_{1}$, then again $\nu(x,\xi)=2\xi_{1}$. This is a model for a Laplacian (again localised where $|\xi_{2}|\sim{}1$) near a curved hypersurface.  Again we solve this on the Fourier side. We require
$$(hD_{\xi_{1}}-\xi_{2}+|\xi'|^{2})f(\xi)=O_{L^{2}}(h)$$
which is satisfied by 
 $$e^{\frac{i}{h}(\xi_{1}(-\xi_{2}+\bar{\xi}^{2})+\frac{1}{3}\xi_{1}^{3})}.$$ To localise this quasimode we place a cut of function in $|\xi|$
$$f(\xi)=\chi^{1}(|\xi|)e^{\frac{i}{h}(\xi_{1}(\xi_{2}-\bar{\xi}^{2})+\frac{1}{3}\xi_{1}^{3})}.$$
Note that $f$ is $L^{2}$ normalised. Again we set
$$u=\mathcal{F}_{h}^{-1}(f).$$
Notice that for this example it is impossible to concentrate a quasimode in the region where $|\nu(x,\xi)|\sim{}h^{\alpha}$. This is due to the curvature which means that the acceleration $\dot{\nu}(x,\xi)=1$, see Proposition \ref{curvedest}. Now
 $$\nu(x,hD)\chi^{2}_{\alpha,\nu}(x,hD)u|_{H}=\frac{1}{(2\pi{}h)^{n/2}}\int{}e^{-\frac{i}{h}(x'\cdot\xi'+\xi_{1}(\xi_{2}-\bar{\xi}^{2})+\frac{1}{3}\xi_{1}^{3})}2\xi_{1}\chi^{2}(h^{-\alpha}|\xi_{1}|)d\xi.$$
We will calculate the $\xi_{1}$ integral via stationary phase. Let
$$\phi(\xi)=-\xi_{1}\xi_{2}+\xi_{1}|\bar{\xi}|^{2}+\frac{1}{3}\xi_{1}^{3}.$$
There is a critical point at
$$\xi_{1}=\sqrt{\xi_{2}-|\bar{\xi}|^{2}}$$
and
$$\partial_{\xi_{1}\xi_{1}}^{2}\phi=2\xi_{2}\sim{}h^{\alpha}.$$
Now the symbol has derivatives no worse than $h^{(1-\alpha)/2}$ so for some $c_{1}\leq{}g(\xi)\leq{}c_{2}$
 $$\nu(x,hD)\chi^{2}_{\alpha,\nu}(x,hD)u=\frac{h^{\alpha}\cdot{}h^{(1-\alpha)/2}\cdot{}h^{-1/2}}{(2\pi{}h)^{(n-1)/2}}\int{}e^{-\frac{i}{h}x'\cdot\xi'+\frac{2}{3}(\xi_{2}-|\bar{\xi}|^{2})^{3/2}}g(\xi)d\xi$$
$$=h^{\frac{\alpha}{2}}\mathcal{F}_{h,n-1}^{-1}[e^{\frac{2i}{3h}(\xi_{2}-|\bar{\xi}|^{2})}g(\xi)]$$
 where $\mathcal{F}_{h,n-1}$ is the $n-1$ dimensional semiclassical Fourier transform. Since $\mathcal{F}_{h,n-1}$ preserves $L^{2}$ norms we have
 $$\norm{\nu(x,hD)\chi^{2}_{\alpha,\nu}(x,hD)u}_{L^{2}(H)}>ch^{\alpha/2}.$$

\end{example}

\begin{example}\label{ex3}
Let $p(x,\xi)=\xi_{2}\xi_{1}$, in this case $\nu(x,\xi)=\xi_{2}$. This symbol does not satisfy the admissibility conditions of \cite{tacy09} to have good restriction bounds (that is it is not Laplace-like). In fact we can construct examples such that
$$\norm{u}_{L^{2}(H)}>h^{-\frac{1}{2}}\norm{u}_{L^{2}(M)}.$$
However these examples require that $\nu(x,\xi)=\xi_{2}\sim{}h$. Again we construct examples on the Fourier side where we must have
$$(\xi_{1}\xi_{2}-1)f_{\alpha}=O_{L^{2}}(h).$$
Let
$$f_{\alpha}=h^{-\frac{1}{2}}\chi^{2}(h^{-\alpha}|\xi_{2}|)\chi^{1}(h^{-1+\alpha}|\xi_{1}|)\chi^{1}(|\bar{\xi}|)$$
and $u_{\alpha}=\mathcal{F}_{h}^{-1}[f_{\alpha}]$.
So
$$\nu(x,hD)u_{\alpha}|_{H}=\frac{h^{-\frac{1}{2}}}{(2\pi{}h)^{n/2}}\int{}e^{\frac{i}{h}x'\cdot{}\xi'}\xi_{2}\chi^{2}(h^{-\alpha}|\xi_{2}|)\chi^{1}(h^{-1+\alpha}|\xi_{1}|)\chi^{1}(|\bar{\xi}|)d\xi.$$
Now for $|x|_{2}<h^{1-\alpha}$ and $|\bar{x}|\leq{}h$ the $e^{\frac{i}{h}x'\cdot\xi'}$ factor does not significantly oscillate. So in this region
$$|\nu(x,hD)u_{\alpha}|>h^{\alpha-\frac{1}{2}-\frac{n}{2}}$$
and so
$$\norm{\nu(x,hD)u_{\alpha}}_{L^{2}(H)}>ch^{\alpha/2}.$$
Note that as $\nu(x,\xi)$ is a function of $\xi$ alone we may define localised all the way down to $\alpha=1$ and these saturating examples continue to hold up to that scale.
\end{example}

\begin{example}\label{ex4}
Finally let $p(x,\xi)=\xi_{1}\xi_{2}-x_{2}$, again $\nu(x,\xi)=\xi_{2}$. Like Example \ref{ex3} this symbol is not admissible as in \cite{tacy09}, however it is curved in the sense that $\dot{\nu}(x,\xi)=-1$. On the Fourier side we need to solve
$$(hD_{\xi_{2}}-\xi_{1}\xi_{2})f=O_{L^{2}}(h).$$
Let
$$f_{\alpha}(\xi)=h^{-\frac{1}{2}+\alpha}\chi^{1}(h^{-1+2\alpha}|\xi_{1}|)e^{\frac{i}{h}\xi_{1}\xi_{2}^{2}}$$
and $u_{\alpha}=\mathcal{F}_{h}^{-1}[f_{\alpha}]$. So
$$\nu(x,hD)u_{\alpha}|_{H}=\frac{h^{-\frac{1}{2}+\alpha}}{(2\pi{}h)^{n/2}}\int{}e^{-\frac{i}{x}(x'\cdot\xi'+\xi_{1}\xi_{2}^{2})}\xi_{2}\chi^{2}(h^{-\alpha}|\xi_{2}|)d\xi.$$
Now for $|x_{2}|<h^{1-\alpha}$ and $|\bar{x}|<h$ the factor
$$e^{-\frac{i}{x}(x'\cdot\xi'+\xi_{1}\xi_{2}^{2})}$$
does not oscillate significantly so
$$|\nu(x,hD)u_{\alpha}|>ch^{\frac{1}{2}+\alpha-\frac{n}{2}}$$
and
$$\norm{\nu(x,hD)u_{\alpha}}_{L^{2}(H)}>ch^{\alpha/2}.$$
\end{example}

The final two examples are not Laplace-like so for these Theorem \ref{thm:tang} only tells us
$${\nu(x,hD)\chi^{1}_{1/3,\nu}(x,hD)}_{L^{2}(H)}\lesssim{}\norm{u}_{L^{2}}.$$
However it appears from these examples that the strong $h^{\alpha/2}$ bounds should still hold in these cases. It is therefore likely that better tangential estimates for these types of operators would be possible from a more fine analysis of the dynamics.  If we write
$$u=\int{}K(x,y)u(y)dy$$
where $K$ is a reproducing kernel and
$$K(x,y)=e^{\frac{i}{h}\phi(x,y)}b(x,y)$$
we can study the $L^{2}$ norm of the restriction of $u$ to $H$ by studying the canonical relation associated with the phase function $\phi$. In examples \ref{ex1} and \ref{ex4} the phase functions
are associated with one-sided folds as studied by Greeleaf-Seeger in \cite{greenleaf94}. Example \ref{ex2} is associated with a two sided fold as in Pan-Sogge \cite{pan90}. This suggests that to obtain sharp results in the tangential setting it would be necessary to classify symbols in terms of the associated canonical relations. In the Laplacian case Galkowski \cite{galkowski} has recently examined the sharp tangential behaviour (dependent of on the curvature of H) but just such an analysis. However were $p(x,\xi)$ is not Laplace-like the tangential question remains open.

\bibliography{references}

\begin{thebibliography}{10}

\bibitem{BLR}
C.~Bardos, G.~Lebeau, and J.~Rauch.
\newblock Sharp sufficient conditions for the observation, control, and
  stabilization of waves from the boundary.
\newblock {\em SIAM J. Control Optim.}, 30(5):1024--1065, 1992.

\bibitem{BGT}
N.~Burq, P.~G{\'e}rard, and N.~Tzvetkov.
\newblock Restrictions of the {L}aplace-{B}eltrami eigenfunctions to
  submanifolds.
\newblock {\em Duke Math. J.}, 138(3):445--486, 2007.

\bibitem{CHT}
H.~Christianson, A.~Hassell, and J.~A. Toth.
\newblock Exterior mass estimates and {$L^2$}-restriction bounds for {N}eumann
  data along hypersurfaces.
\newblock {\em Int. Math. Res. Not. IMRN}, (6):1638--1665, 2015.

\bibitem{galkowski}
J.~Galkowski.
\newblock The {$L^{2}$} behavior of eigenfunctions near the glancing set.
\newblock {\em arXiv:1604.01699v1}, 2016.

\bibitem{GL93}
P.~G{\'e}rard and {\'E}.~Leichtnam.
\newblock Ergodic properties of eigenfunctions for the {D}irichlet problem.
\newblock {\em Duke Math. J.}, 71(2):559--607, 1993.

\bibitem{greenleaf94}
A.~Greenleaf and A.~Seeger.
\newblock Fourier integral operators with fold singularities.
\newblock {\em J. Reine Angew. Math.}, 455:35--56, 1994.

\bibitem{HTacy}
A.~Hassell and M.~Tacy.
\newblock Semiclassical {$L^p$} estimates of quasimodes on curved
  hypersurfaces.
\newblock {\em J. Geom. Anal.}, 22(1):74--89, 2012.

\bibitem{HTao02}
A.~Hassell and T.~Tao.
\newblock Upper and lower bounds for normal derivatives of {D}irichlet
  eigenfunctions.
\newblock {\em Math. Res. Lett.}, 9(2-3):289--305, 2002.

\bibitem{HTao10}
A.~Hassell and T.~Tao.
\newblock Erratum for ``{U}pper and lower bounds for normal derivatives of
  {D}irichlet eigenfunctions''.
\newblock {\em Math. Res. Lett.}, 17(4):793--794, 2010.

\bibitem{Hu}
R.~Hu.
\newblock {$L^p$} norm estimates of eigenfunctions restricted to submanifolds.
\newblock {\em Forum Math.}, 21(6):1021--1052, 2009.

\bibitem{pan90}
Y.~Pan and C.~Sogge.
\newblock Oscillatory integrals associated to folding canonical relations.
\newblock {\em Colloq. Math.}, 60/61(2):413--419, 1990.

\bibitem{R40}
F.~Rellich.
\newblock Darstellung der {E}igenwerte von {$\Delta u+\lambda u=0$} durch ein
  {R}andintegral.
\newblock {\em Math. Z.}, 46:635--636, 1940.

\bibitem{tacy09}
M.~Tacy.
\newblock Semiclassical {$L^p$} estimates of quasimodes on submanifolds.
\newblock {\em Comm. Partial Differential Equations}, 35(8):1538--1562, 2010.

\bibitem{tataru98}
D.~Tataru.
\newblock On the regularity of boundary traces for the wave equation.
\newblock {\em Ann. Scuola Norm. Sup. Pisa Cl. Sci. (4)}, 26(1):185--206, 1998.

\bibitem{Zworski12}
M.~Zworski.
\newblock {\em Semiclassical analysis}, volume 138 of {\em Graduate Studies in
  Mathematics}.
\newblock American Mathematical Society, Providence, RI, 2012.

\end{thebibliography}
\bibliographystyle{abbrv} 

\end{document}